\documentclass [12pt, a4paper]{article}
\usepackage [english] {babel}
\usepackage [applemac] {inputenc}
\usepackage[charter]{mathdesign}
\usepackage [T1] {fontenc}
\usepackage[scale=0.75]{geometry}
\usepackage{enumitem}
\usepackage[square,numbers]{natbib}
\usepackage{tocloft}
\usepackage{mathtools}
\usepackage{amsthm, amsmath}
\usepackage{pifont}
\usepackage{pgf,tikz,tikz-cd}
\usepackage{indentfirst}
\usepackage{multirow}
\usetikzlibrary{trees, matrix, arrows, calc, patterns, decorations.pathreplacing}
\usepackage[pdftex,pdfauthor={Fran\c{c}ois Thilmany},pdftitle={Lattices of minimal covolume in SL(n,R)}, bookmarksopen, bookmarksnumbered, colorlinks={false}, urlcolor={cyan}, citecolor={blue}]{hyperref}

\swapnumbers
\theoremstyle{plain}
\newtheorem{thm}{Theorem}[section]
\newtheorem*{thm*}{Theorem}
\newtheorem{lem}[thm]{Lemma}
\newtheorem*{lem*}{Lemma}

\newtheorem*{cor*}{Corollary}

\newtheorem*{prop*}{Proposition}

\theoremstyle{definition}
\newtheorem{d?f}[thm]{Definition}
\newtheorem*{d?f*}{Definition}

\newtheorem*{rem*}{Remark}

\newtheorem*{ex*}{Example}
\newtheorem{tabl}[thm]{Table}
\newtheorem{tabl*}{Table}

\newcommand \A {\mathbb{A}}%
\newcommand \C {\mathbb{C}}%
\newcommand \F {\mathbb{F}}%
\newcommand \N {\mathbb{N}}%
\newcommand \Q {\mathbb{Q}}%
\newcommand \R {\mathbb{R}}%
\newcommand \Z {\mathbb{Z}}%

\newcommand \cC {\mathcal{C}}%
\newcommand \cG {\mathcal{G}}%
\newcommand \cM {\mathcal{M}}%
\newcommand \cS {\mathcal{S}}%

\newcommand \rA {\mathrm{A}}%
\newcommand \rM {\mathrm{M}}%

\newcommand \ff {\mathfrak{f}}%
\newcommand \fs {\mathfrak{s}}%
\newcommand \fd {\mathfrak{d}}%
\newcommand \fp {\mathfrak{p}}%
\newcommand \fP {\mathfrak{P}}%
\newcommand \fD {\mathfrak{D}}%

\newcommand \im {\operatorname{im}}%
\newcommand \Aut {\operatorname{Aut}}%
\newcommand \SL {\operatorname{SL}}%
\newcommand \GL {\operatorname{GL}}%
\newcommand \SU {\operatorname{SU}}%
\newcommand \SO {\operatorname{SO}}%
\newcommand \rH {\operatorname{H}}%
\newcommand \rR {\operatorname{R}}%
\newcommand{\ol}[1]{\overline{#1}}%
\newcommand{\ot}{\otimes}%
\DeclarePairedDelimiter{\ceil}{\lceil}{\rceil}

\begin{document}

\title{Lattices of minimal covolume in $\SL_n(\R)$}
\author{François Thilmany\thanks{The author is supported by the Fonds National de la Recherche, Luxembourg (AFR grant 11275005)}}
\date{}
\maketitle{}

\begin{abstract}
The objective of this paper is to determine the lattices of minimal covolume in $\SL_n(\R)$, for $n \geq 3$. The answer turns out to be the simplest one: $\SL_n(\Z)$ is, up to automorphism, the unique lattice of minimal covolume in $\SL_n(\R)$. In particular, lattices of minimal covolume in $\SL_n(\R)$ are non-uniform when $n \geq 3$, contrasting with Siegel's result for $\SL_2(\R)$. This answers for $\SL_n(\R)$ the question of Lubotzky: is a lattice of minimal covolume typically uniform or not?
\end{abstract}

\tableofcontents

%%%%%%%%%%%%%%%%%%%%

\section{Introduction}

\subsection{A brief history}

The study of lattices of minimal covolume in $\SL_n$ originated with Siegel's work \cite{Siegel45} on $\SL_2(\R)$. Siegel showed that in $\SL_2(\R)$, a lattice of minimal covolume is given by the $(2,3,7)$-triangle group. He raised the question to determine which lattices attain minimum covolume in groups of isometries of higher-dimensional hyperbolic spaces. For $\SL_2(\C)$, which acts on hyperbolic 3-space, the minimum among non-uniform lattices was established by Meyerhoff \cite{Meyerhoff85}; among all lattices in $\SL_2(\C)$, the minimum was exhibited more recently by Gehring, Marshall and Martin \cite{GehringMartin09,MarshallMartin12}, and is attained by a uniform lattice. 

Lubotzky established the analogous result \cite{Lubotzky90} for $\SL_2\left(\F_q((t^{-1}))\right)$, where this time $\SL_2\left(\F_q[t]\right)$ attains the smallest covolume. Lubotzky observed that in this case, as opposed to the $(2,3,7)$-triangle group in $\SL_2(\R)$, the lattice of minimal covolume is not uniform; he then asked whether, for a lattice of minimal covolume in a semi-simple Lie group, the typical situation is to be uniform, or not. 

Progress has been made on this question, and Salehi Golsefidy showed \cite{Salehi09} that for most Chevalley groups $G$ of rank at least 2, $G\left(\F_q[t]\right)$ is the unique (up to isomorphism) lattice of minimal covolume in $G\left( \F_q((t^{-1}))\right)$. Salehi Golsefidy also obtained \cite{Salehi13} that for most simply connected almost simple groups over $\F_q((t^{-1}))$, a lattice of minimal covolume will be non-uniform (provided Weil's conjecture on Tamagawa numbers holds). 

On the other side of the picture, when the rank is 1, Belolipetsky and Emery \cite{Belolipetsky04, BelolipetskyEmery12} determined the lattices of minimal covolume among arithmetic lattices in \linebreak $\SO(n,1)(\R)$ ($n \geq 4$) and showed that they are non-uniform. For $\SU(n,1)(\R)$, Emery and Stover \cite{EmeryStover14} determined the lattices of minimal covolume among the non-uniform arithmetic ones, but to the best of the author's knowledge, this has not been compared to the uniform arithmetic ones in this case. 
Unfortunately, in the rank 1 case, it is not known whether a lattice of minimal covolume is necessarily arithmetic.

The above results give a partial answer to the question of Lubotzky in these two respective situations. In this paper, we intend to contribute to the question for $\SL_n(\R)$. We show that, up to automorphism, the non-uniform lattice $\SL_n(\Z)$ is the unique lattice of minimal covolume in $\SL_n(\R)$.

%%%%%

\subsection{Outline}

The goal of the present paper is to prove the following theorem.
\begin{thm*}
Let $n \geq 3$ and let $\Gamma$ be a lattice of minimal covolume for some (any) Haar measure in $\SL_n(\R)$. Then $\sigma(\Gamma) = \SL_n(\Z)$ for some (algebraic) automorphism $\sigma$ of $\SL_n(\R)$. 
\end{thm*}

The argument relies in an indispensable way on the important work of Prasad \cite{Prasad89} and Borel and Prasad \cite{BorelPrasad89} (there will be multiple references to results contained in these two articles). We will proceed as follows.

We start with a lattice $\Gamma$ of minimal covolume in $\SL_n(\R)$. Using Margulis' arithmeticity theorem and Rohlfs' maximality criterion, we find a number field $k$, an archimedean place $v_0$ and a simply connected absolutely almost simple $k$-group $G$ for which $\Gamma$ is identified with the normalizer of a principal arithmetic subgroup $\Lambda$ in $G(k_{v_0})$. The latter means that there is a collection of parahoric subgroups $\{P_v\}_{v \in V_f}$ such that $\Lambda$ consists precisely of the elements of $G(k)$ whose image in $G(k_v)$ lies in $P_v$ for all $v \in V_f$. This allows us to express the covolume of $\Gamma$ as $\mu(G(k_{v_0}) / \Gamma) = [\Gamma: \Lambda]^{-1}\; \mu(G(k_{v_0}) / \Lambda)$.  

The factor $\mu(G(k_{v_0}) / \Lambda)$ can be computed using Prasad's volume formula \cite{Prasad89}, and the result depends on the arithmetics of $k$ and of the parahorics $P_v$, as well as on the quasi-split inner form of $G$. 

On the other hand, the index $[\Gamma: \Lambda]$ can be controlled using techniques developed by Rohlfs \cite{Rohlfs79}, and Borel and Prasad \cite{BorelPrasad89}. The bound depends namely on the first Galois cohomology group of the center of $G$ and on its action on the types of the parahorics $P_v$. 

Once we have an estimate on the covolume of $\Gamma$, we can compare it to the covolume of $\SL_n(\Z)$ in $\SL_n(\R)$. We argue that for the former not to exceed the latter, it must be that $k$ is $\Q$, $G$ is an inner form of $\SL_n$, and all the parahorics are hyperspecial. This is carried out in sections \ref{kisQ}-\ref{hyperspecial}. 

Finally, using local-global techniques, we conclude that $\Gamma$ must be the image of $\SL_n(\Z)$ under some automorphism of $\SL_n(\R)$.

%%%%%

\subsection{Acknowledgements}

First of all, the author wishes wholeheartedly to thank Alireza Salehi Golsefidy, to whom the author is greatly indebted for suggesting the subject of the present work, for his precious insight regarding some of the key points at issue, and for his helpful remarks throughout the completion of this project. 

The author also wishes to thank Mikhail Belolipetsky and Gopal Prasad for their very helpful comments, and Jake Postema for interesting conversations about some of the number-theoretical aspects of this work. 

Lastly, the author is very grateful for the financial support of the Fonds National de la Recherche du Luxembourg, that allowed him to devote full attention to this project.

%%%%%

\subsection{Notation and preliminaries}

The contents of the paper will assume familiarity with the theory of algebraic groups, Bruhat-Tits theory and basic number theory. We refer the reader to \cite{PlatonovRapinchuk94} for an exposition of some of these topics and a more complete list of the available literature. \bigbreak

As much as possible, we will follow the notation adopted by Borel and Prasad in \cite{Prasad89} and \cite{BorelPrasad89}. 

\begin{enumerate}[label=$\bullet$,itemsep=0pt]
\item $\N$, $\Q$, $\R$, $\C$ respectively denote the sets of strictly positive natural, rational, real and complex numbers. For $p$ a place or a prime, $\Q_p$ denotes the field of $p$-adic numbers and $\Z_p$ its ring of $p$-adic integers. $\F_p$ denotes the finite field with $p$ elements.
\item In what is to follow, we will fix a number field $k$ of degree $m$, and $V$, $V_\infty$ and $V_f$ will always denote the set of places, archimedean places and non-archimedean places of $k$. We will always normalize each non-archimedean place $v$ so that $\im v = \Z$. 
\item For $v \in V$, $k_v$ will denote the $v$-adic completion of $k$. For $v \in V_f$, $\widehat{k}_v$ is the maximal unramified extension of $k_v$, $\ff_v$ denotes the residue field of $k$ at $v$ and $q_v = \# \ff_v$ is the cardinality of the latter. 
\item $\A_k$ denotes the ring of adeles of $k$, and the adeles of $\Q$ will be abbreviated $\A$. 
\item When working with the adele points $G(\A_k)$ (or variations of them, e.g.~finite adeles) of an algebraic group $G$, we will freely identify $G(k)$ with its image in $G(\A_k)$ under the diagonal embedding, and vice-versa. 
\item For $l$ a finite extension of $k$, we denote $D_l$ the absolute value of the discriminant of $l$ (over $\Q$) and $\fd_{l/k}$ the relative discriminant of $l$ over $k$; $h_l$ is the class number of $l$. The units of $l$ will be denoted by $U_l$, and the subgroup of roots of unity in $l$ by $\mu(l)$. 
\item $G$ will be a simply connected absolutely almost simple group (of type $\mathrm{A}_r$) defined over $k$. We denote $r = n-1$ its absolute rank, and for $v \in V_f$, $r_v$ is its rank over $\widehat{k}_v$. 
\item $\cG$ denotes the quasi-split inner $k$-form of $G$, $l$ will denote its splitting field. 
\item $\SU_n$ denotes the special unitary group defined over $\R$ associated to the positive-definite hermitian form $|z_1| + \dots + |z_n|$ on $\C^n$. Its group $\SU_n(\R)$ of real points is the usual special unitary group, the unique compact connected simply connected almost simple Lie group of type $\rA_{n-1}$. 
\item $\zeta$ denotes Riemann's zeta function. 
\item For $n \in \Z$, we set $\tilde{n} = 1$ or 2 if $n$ is respectively odd or even.
\item For $x \in \R$, $\ceil{x}$ denotes the ceiling of $x$, that is the smallest integer $n$ such that $n \geq x$. 
\item $V_n$ will denote the quantity $\prod_{i=1}^{n-1} \frac{i!}{(2\pi)^{i+1}}$. 
\end{enumerate}

%%%%%%%%%%

\section{The setting} \label{setup}

On $\SL_n$, we pick a left-invariant exterior form $\omega_0$ of highest degree which is defined over $\Q$. The form $\omega_0$ induces a left-invariant form on $\SL_n(\R)$, also to be denoted $\omega_0$, which in turn induces a left-invariant form on $\SU_n(\R)$ through their common Lie algebra. Let $c_0 \in \R$ be such that $\SU_n(\R)$ has volume 1 for the Haar measure determined in this way by $c_0 \omega_0$; we denote $\mu_0$ the Haar measure given by $c_0 \omega_0$ on $\SL_n(\R)$. 

Computing the covolume of $\SL_n(\Z)$ goes back to Siegel \cite{Siegel45'}, and for this particular measure, it is given by
$$\mu_0(\SL_n(\R) / \SL_n(\Z)) = \left( \prod_{i=1}^r \frac{i!}{(2\pi)^{i+1}} \right) \cdot \prod_{i=2}^{n} \zeta(i).$$
(To obtain this, one can for example use \cite[thm. 3.7]{Prasad89}; see §\ref{volumeformula} below. For the lattice $\Lambda = \SL_n(\Z)$, one can take $P_v = \SL_n(\Z_{v})$, so that $e(P_v) = \frac{(q_v-1)q_v^{n^2-1}}{\prod_{i=0}^{n-1} (q_v^n - q_v^i)} = \prod_{i=2}^n \frac{1}{1 - q_v^{-i}}$ and $\prod_{v \in V_f} e(P_v) = \prod_{i=2}^n \zeta(i)$.) \bigbreak

Let $\Gamma$ be a lattice of minimal covolume for $\mu_0$ in $\SL_n(\R)$ (the existence of such a lattice can be obtained using the Kazhdan-Margulis theorem, see for example \cite{Wang72}); in particular, $\Gamma$ is a maximal lattice. By Margulis' arithmeticity theorem \cite{Margulis91} and Rohlfs' maximality criterion \cite[prop. 1.4]{BorelPrasad89} combined, there is a number field $k$, a place $v_0 \in V_\infty$, a simply connected absolutely almost simple group $G$ defined over $k$, and a parahoric subgroup $P_v$ of $G(k_v)$ for each $v \in V_f$, such that:

\begin{enumerate}[label=(\roman*),itemsep=0pt]
\item $k_{v_0} = \R$,

\item there is an isomorphism $\iota : \SL_n \rightarrow G$ defined over $k_{v_0}$ (in particular, $\SL_n(\R) \cong G(k_{v_0})$),

\item the collection $\{P_v\}_{v \in V_f}$ is coherent, i.e. $\prod_{v \in V_\infty} G(k_v) \times \prod_{v \in V_f} P_v$ is an open subgroup of the adele group $G(\A_k)$,

\item $\iota(\Gamma)$ is the normalizer of the lattice $\Lambda = G(k) \cap \iota(\Gamma)$ in $G(k_{v_0})$, and $\Lambda = G(k) \cap \prod_{v \in V_f} P_v$ is the \emph{principal arithmetic subgroup} determined by the collection $\{P_v\}_{v \in V_f}$. 
\end{enumerate}

This already imposes the signature of $k$ and of the splitting field $l$ of the quasi-split inner form $\cG$ of $G$. Indeed, for any archimedean place $v \neq v_0$, the group $G(k_v)$ must be compact (otherwise $\Lambda$ would be dense in $G(k_{v_0})$ by strong approximation). In consequence, $k_v \cong \R$ for $v \in V_\infty - \{v_0\}$ (otherwise $G(k_v) \cong \SL_n(\C)$ is not compact) and $k$ is totally real. Note that in fact, for each $v \in V_\infty - \{v_0\}$, $G(k_v)$ is isomorphic to $\SU_n(\R)$, the unique compact connected simply connected almost simple Lie group of type $\rA_{n-1}$. 

Recall that since $G$ is of type $\mathrm{A}$, either $l = k$ or $l$ is a quadratic extension of $k$. Regardless, if $v \in V_\infty - \{v_0\}$, it may not be that $l$ embeds into $k_v$: indeed, if this happens, then $\cG$ splits over $k_v$, and thus $G$ would be an inner $k_v$-form of $\SL_n$. This prohibits $G(k_v)$ from being compact, as inner $k_v$-forms of $\SL_n$ are isotropic when $n \geq 3$. Thus, in the former case, when $G$ is an inner $k$-form, it must be that $V_\infty - \{v_0\}$ is empty, i.e.~$l = k = \Q$. In the latter case, when $G$ is an outer $k$-form, for each $v \in V_\infty - \{v_0\}$ the real embedding $k \to k_v$ extends to two (conjugate) complex embeddings of $l$. On the other hand, $G$, hence $\cG$, splits over $k_{v_0}$, thus $l$ embeds in $k_{v_0}$. Combined, we see in this case that the signature of $l$ is $(2, m-1)$. \bigbreak

On $G$, we pick a left-invariant exterior form $\omega$ of highest degree which is defined over $k$. The form $\omega$ induces a left-invariant form on $G(k_{v_0})$, also to be denoted $\omega$, which in turn induces a left-invariant form on $\SU_n(\R)$ through their common Lie algebra. Let $c \in \R$ be such that $\SU_n(\R)$ has volume 1 for the Haar measure determined in this way by $c \omega$; we denote $\mu$ the Haar measure determined by $c \omega$ on $G(k_{v_0})$. By construction, $\mu$ agrees with the measure induced from $\mu_0$ through the isomorphism $\iota$. In what follows, we will freely identify $\SL_n(\R)$ with $G(k_{v_0})$, $\Gamma$ with its image $\iota(\Gamma)$ and $\mu_0$ with $\mu$. With this, we have
$$\mu_0(\SL_n(\R) / \Gamma) = \mu(G(k_{v_0}) / \Gamma) = [\Gamma: \Lambda]^{-1}\; \mu(G(k_{v_0}) / \Lambda).$$

%%%%%%%%%%

\section{Prasad's volume formula} \label{volumeformula}

We fix a left-invariant exterior form $\omega_{qs}$ defined over $k$ on the quasi-split inner $k$-form $\cG$ of $G$. As before, $\omega_{qs}$ induces for each $v \in V_\infty$ an invariant form on $\cG(k_v)$, and in turn on any maximal compact subgroup of $\cG(\C)$ through their common Lie algebra. (Note again that such a maximal compact subgroup can be identified with $\SU_n(\R)$.) For each $v \in V_\infty$, we choose $c_v \in k_v$ such that the corresponding maximal compact subgroup has measure 1 for the Haar measure determined in this way by $c_v \omega_{qs}$. 

Let $\varphi: G \rightarrow \cG$ be an isomorphism, defined over some Galois extension $K$ of $k$, such that $\varphi^{-1} \circ {}^\gamma \varphi$ is an inner automorphism of $G$ for all $\gamma$ in the Galois group of $K$ over $k$. Then $\varphi$ induces an invariant form $\omega^* = \varphi^* (\omega_{qs})$ on $G$, defined over $k$. Once again, $\omega^*$ induces for each $v \in V_\infty$ a form on $G(k_v)$ and then a form on any maximal compact subgroup of $G(\C)$ through their Lie algebras. It turns out \cite[§3.5]{Prasad89} that the volume of any such maximal compact subgroup for the Haar measure determined in this way by $c_v \omega^*$ is 1. This implies in particular that the Haar measure determined on $G(k_{v_0})$ by $c_{v_0} \omega^*$ is actually the measure $\mu$ that we constructed earlier. 

For each $v \in V_\infty$, we endow $G(k_v)$ with the Haar measure $\mu_v$ determined by $c_v \omega^*$. As we observed, $\mu_{v_0} = \mu$, and for $v \in V_\infty - \{v_0\}$, $G(k_v)$ is compact, hence $\mu_v(G(k_v)) = 1$ by definition of $\mu_v$. 
The product $G_\infty = \prod_{v \in V_\infty} G(k_v)$ is then endowed with the product measure $\mu_\infty = \prod_{v \in V_\infty} \mu_v$. 
The lattice $\Lambda$ embeds diagonally in $G_\infty$; we will abusively denote its image by $\Lambda$ as well. If $F$ is a fundamental domain for $\Lambda$ in $G(k_{v_0})$, then $F_\infty = F \times \prod_{v \in V_\infty - \{v_0\}} G(k_v)$ is a fundamental domain for $\Lambda$ in $G_\infty$. Therefore
$$\mu_\infty(G_\infty / \Lambda) = \mu_\infty(F_\infty) = \mu_{v_0}(F) \cdot \prod_{v \in V_\infty - \{v_0\}} \mu(G(k_{v})) = \mu(G(k_{v_0}) / \Lambda).$$

Using this observation, the main result from \cite{Prasad89} allows us to compute 
\begin{equation*}
\mu(G(k_{v_0})/ \Lambda) = D_k^{\frac{1}{2}\dim G} (D_l / D_k^{[l:k]})^{\frac{1}{2}\fs(\cG)} \left( \prod_{i=1}^r \frac{i!}{(2\pi)^{i+1}} \right)^{[k:\Q]} \prod_{v \in V_f} e(P_v). \label{volume} \tag{V}
\end{equation*}
Here, $l$ is the splitting field of the quasi-split inner $k$-form $\cG$ of $G$ ($l$ is $k$ or a quadratic extension of $k$), $r=n-1$ is the absolute rank of $G$, $\fs(\cG) = 0$ if $\cG$ is split, otherwise $\fs(\cG) = \frac{1}{2}r(r+3)$ if $r$ is even or $\fs(\cG) = \frac{1}{2}(r-1)(r+2)$ if $r$ is odd, and $e(P_v) = \frac{q_v^{(\dim \overline{M}_v + \dim \overline{\cM}_v)/2}}{\#\overline{M}_v(\ff_v)}$ is the inverse of the volume of $P_v$ for a particular measure. 
We refer to \cite{Prasad89} for the unexplained notation (in the present setting, $S = V_\infty$ consists only of real places). 

%%%%%%%%%%

\section{An upper bound on the index} \label{indexbound}

For the convenience of the reader, we briefly recollect the upper bound on the index $[\Gamma: \Lambda]$ developed by Borel and Prasad. The complete exposition, proofs and references are to be found in \cite[§2 \& §5]{BorelPrasad89} (in the present setting, $\cS = \{v_0\}$, $G'=G$, $\Gamma' = \Gamma$, etc.). \bigbreak

For each place $v \in V_f$, we fix a maximal $k_v$-split torus $T_v$ of $G$; we also fix an Iwahori subgroup $I_v$ of $G(k_v)$ such that the chamber in the affine building of $G({k}_v)$ fixed by ${I}_v$ is contained in the apartment corresponding to $T_v$. We denote by $\Delta_v$ the basis determined by $I_v$ of the affine root system of $G(k_v)$ relative to $T_v$. 

$\Aut(G(k_v))$, hence also the adjoint group $\overline{G}(k_v)$, acts on $\Delta_v$; we denote by $\xi_v: \overline{G}(k_v) \rightarrow \Aut(\Delta_v)$ the corresponding morphism. Let $\Xi_v$ be the image of $\xi_v$. \bigbreak

Let $C$ be the center of $G$ and $\varphi: G \rightarrow \overline{G}$ the natural central isogeny, so that there is an exact sequence of algebraic groups
$$1 \rightarrow C \rightarrow G \xrightarrow{\varphi} \overline{G} \rightarrow 1.$$
This sequence gives rise to long exact sequences (of pointed sets), which we store in the following commutative diagram ($v \in V$). 
\begin{equation*}
\begin{tikzcd}
1 \arrow{r} & C(k)\arrow{r}\arrow{d} & G(k) \arrow{r}{\varphi}\arrow{d} & \overline{G}(k) \arrow{r}{\delta}\arrow{d} & \rH^1(k,C) \arrow{r}\arrow{d} & \rH^1(k,G) \arrow{d}\\
1 \arrow{r} & C(k_v)\arrow{r} & G(k_v) \arrow{r}{\varphi} & \overline{G}(k_v) \arrow{r}{\delta_v} & \rH^1(k_v,C) \arrow{r} & \rH^1(k_v,G)
\end{tikzcd} \label{long} \tag{$\mathrm{L}_v$}
\end{equation*}
When $v \in V_f$, we have that $\rH^1(k_v, G) = 1$ by a result of Kneser \cite{Kneser65} and thus $\delta_v$ induces an isomorphism 
$$\overline{G}(k_v) / \varphi(G(k_v)) \cong \rH^1(k_v, C).$$
Recall that $\xi_v$ is trivial on $\varphi(G(k_v))$. Thus $\xi_v$ induces a map $\rH^1(k_v, C) \rightarrow \Xi_v$, which we abusively denote by $\xi_v$ as well.

Let $\Delta = \prod_{v \in V_f} \Delta_v$, $\Xi = \bigoplus_{v \in V_f} \Xi_v$ and $\Theta = \prod_{v \in V_f} \Theta_v$, where $\Theta_v \subset \Delta_v$ is the type of the parahoric $P_v$ associated to $\Lambda$. $\Xi$ acts on $\Delta$ componentwise, and we denote by $\Xi_{\Theta_v}$ the stabilizer of $\Theta_v$ in $\Xi_v$ and $\Xi_\Theta$ the stabilizer of $\Theta$ in $\Xi$. 
The morphisms $\xi_v$ induce a map
$$\xi: \rH^1(k,C) \rightarrow \Xi: c \mapsto \xi(c) = (\xi_v(c_v))_{v \in V_f}$$
where $c_v$ denotes the image of $c$ in $\rH^1(k_v, C)$. With this, we define
\begin{align*}
\rH^1(k,C)_\Theta &= \{ c \in \rH^1(k,C) \mid \xi(c) \in \Xi_\Theta \} \\
\rH^1(k,C)'_\Theta &= \{ c \in \rH^1(k,C)_\Theta \mid c_{v_0} = 1 \} \\
\rH^1(k,C)_\xi &=  \{ c \in \rH^1(k,C) \mid \xi(c) = 1 \}. 
\end{align*}

Borel and Prasad \cite[prop.~2.9]{BorelPrasad89} use the exact sequence due to Rohlfs
$$1 \rightarrow C(k_{v_0}) / (C(k) \cap \Lambda) \rightarrow \Gamma / \Lambda \rightarrow \delta(\ol{G}(k)) \cap \rH^1(k,C)'_{\Theta} \rightarrow 1. $$
Since $k_{v_0} = \R$, $C(k_{v_0}) = \{1\}$ or $\{1,-1\}$ depending whether $n$ is odd or even. In particular, it follows that $C(k_{v_0}) = C(k) \cap \Lambda$ and $\Gamma / \Lambda \cong \delta(\ol{G}(k)) \cap \rH^1(k,C)'_{\Theta}$. 
Also, it is clear that the kernel of $\xi$ restricted to $\delta(\ol{G}(k)) \cap \rH^1(k,C)'_\Theta$ is contained in $\delta(\ol{G}(k)) \cap \rH^1(k,C)_\xi$, implying that $\# \left( \delta(\ol{G}(k)) \cap \rH^1(k,C)'_\Theta \right) \leq \# \left( \delta(\ol{G}(k)) \cap \rH^1(k,C)_\xi \right) \cdot \prod_{v \in V_f} \# \Xi_{\Theta_v}$, and in turn,
\begin{equation*}
[\Gamma: \Lambda] \leq \# \left(\delta(\ol{G}(k)) \cap \rH^1(k,C)_\xi \right) \cdot \prod_{v \in V_f} \# \Xi_{\Theta_v} \leq \# \rH^1(k,C)_\xi \cdot \prod_{v \in V_f} \# \Xi_{\Theta_v}. \tag{I} \label{index}
\end{equation*}

In the next two subsections, we try to control the size of $\delta(\ol{G}(k)) \cap \rH^1(k,C)_\xi$. We distinguish the case where $G$ is an inner $k$-form of $\SL_n$ from the case $G$ is an outer $k$-form. For the former, we follow the argument of \cite[prop.~5.1]{BorelPrasad89}. In the latter, we will adapt to our setting a refinement of the bounds of Borel and Prasad due to Mohammadi and Salehi Golsefidy \cite[§4]{MohammadiSalehi12}. Except for minor modifications, all the material in this section can be found in these two sources. 

\subsection{The inner case} \label{indexinner}

Although in the inner case we have already established that $k = \Q$, we will discuss it for an arbitrary (totally real) field $k$, as this will be useful to treat the outer case as well. Let us thus assume $G$ is an inner $k$-form, i.e.~(by the classification) $G$ is isomorphic to $\SL_{n'} \fD$ for some central division algebra $\fD$ over $k$ of index $d = n/n'$. Similarly, over $k_v$, $G$ is isomorphic to $\SL_{n_v} \fD_v$ for some central division algebra $\fD_v$ over $k_v$ of index $d_v = n / n_v$. The center $C$ of $G$ is isomorphic to $\mu_n$, the kernel of the map $\GL_1 \rightarrow \GL_1: x \mapsto x^n$, and thus for any field extension $K$ of $k$, $\rH^1(K, C)$ may (and will in this paragraph) be identified with $K^\times / K^{\times n}$ (where $K^{\times n} = \{ x^n \mid x \in K^\times\}$). With this identification, the canonical map $\rH^1(k,C) \rightarrow \rH^1(k_v, C)$ corresponds to the canonical map $k^\times / k^{\times n} \rightarrow k_v^\times / k_v^{\times n}$. 

The action of $\rH^1(k_v, C)$ on $\Delta_v$ can be described as follows: $\Delta_v$ is a cycle of length $n_v$, on which $\overline{G}(k_v)$ acts by rotations, i.e. $\Xi_v$ can be identified with $\Z / n_v \Z$. The action of $\rH^1(k_v, C)$ is then given by the morphism 
$$k_v^\times / k_v^{\times n} \rightarrow \Z/ n_v \Z : x \mapsto v(x) \mod n_v.$$
From this description, we see that $x \in k_v^\times / k_v^{\times n}$ acts trivially on $\Delta_v$ precisely when $v(x) \in n_v \Z$; in particular, if $G$ splits over $k_v$, $x$ acts trivially if and only if $v(x) \in n \Z$. We can form the exact sequence
$$1 \rightarrow k_n / k^{\times n} \rightarrow \rH^1(k,C)_\xi \xrightarrow{(v)_{v \in V_f}} \bigoplus_{v \in V_f} \Z / n \Z,$$
where $k_n = \{ x \in k^\times \mid v(x) \in n\Z \textrm{ for all } v \in V_f\}$. By the above, the image of $\rH^1(k,C)_\xi$ lies in the subgroup $\bigoplus_{v \in V_f} n_v \Z / n \Z$. 
Let $T$ be the set of places $v \in V_f$ where $G$ does not split over $k_v$, i.e.~for which $n_v \neq n$. Then the exact sequence yields 
$$\# \rH^1(k,C)_\xi \leq \#(k_n / k^{\times n}) \cdot \prod_{v \in T} d_v.$$
The proof of \cite[prop.~0.12]{BorelPrasad89} shows that $\#(k_n / k^{\times n}) \leq h_k \tilde{n} n^{[k:\Q]-1}$, where $\tilde{n} = 1$ or 2 if $n$ is respectively odd or even. In the case $k = \Q$, which will be of interest later, it is indeed clear that $\#(\Q_n / \Q^{\times n}) = \tilde{n}$. 

\subsection{The outer case} \label{indexouter}

Second, we assume $G$ is an outer $k$-form. The centers of $G$ and of the quasi-split inner form $\cG$ of $G$ are $k$-isomorphic, hence there is an exact sequence
\begin{equation}
1 \rightarrow C \rightarrow \rR_{l/k}(\mu_n) \xrightarrow{N} \mu_n \rightarrow 1, \label{outercenter1} 
\end{equation}
where $\mu_n$ denotes the kernel of the map $\GL_1 \rightarrow \GL_1: x \mapsto x^n$ as above, $\rR_{l/k}$ denotes the restriction of scalars from $l$ to $k$, and $N$ is (induced by) the norm map of $l/k$. The long exact sequence associated to it yields
\begin{equation}
1 \to \mu_n(k) / N(\mu_n(l)) \to \rH^1(k,C) \to l_0 / l^{\times n} \to 1 \label{outercenter2} 
\end{equation}
where $l_0 / l^{\times n}$ denotes the kernel of the norm map $N: l^\times / l^{\times n} \to k^\times / k^{\times n}$. The Hasse principle for simply connected groups allows us to write
\begin{equation}
\begin{tikzcd}
\overline{G}(k) \arrow{r}{\delta}\arrow{d} & \rH^1(k,C) \arrow{r}\arrow{d} & \rH^1(k,G) \arrow{d}{\rotatebox{90}{\(\sim\)}}\\
\prod_{v \in V_\infty} \overline{G}(k_v) \arrow{r}{(\delta_v)_{v}} & \prod_{v \in V_\infty} \rH^1(k_v,C) \arrow{r} & \prod_{v \in V_\infty} \rH^1(k_v,G). \label{hasseprinciple}
\end{tikzcd}
\end{equation}

If $n$ is odd, we can make the following simplifications: $\mu_n(k) = \{1\}$ and thus $\rH^1(k,C) \cong l_0 / l^{\times n}$ in \eqref{outercenter2}; using the analogous sequence for $k_v$, we also have $\rH^1(k_v,C) \cong \{1\}$ for $v \in V_\infty$. Thus, in \eqref{hasseprinciple}, we read that $\delta$ is surjective and conclude $\delta(\overline{G}(k)) \cong l_0 / l^{\times n}$. 

If $n$ is even, a weaker conclusion holds provided $l$ has at least one complex place, i.e.~if $V_\infty \neq \{v_0\}$. Indeed, if $v_1 \in V_\infty - \{v_0\}$, so that $l \ot_k k_{v_1} = \C$, then $(l \ot_k k_{v_1})^\times / (l \ot_k k_{v_1})^{\times n} = \{1\}$ and the long exact sequences associated to \eqref{outercenter1} read
\begin{equation}
\begin{tikzcd}
1 \arrow{r} & \{\pm 1\} \arrow{r}\arrow{d}{\rotatebox{90}{\(\sim\)}} & \rH^1(k,C) \arrow{r}\arrow{d} & l_0 / l^{\times n} \arrow{r}\arrow{d} & 1 \\
1 \arrow{r} & \{\pm 1\} \arrow{r}{\sim} & \rH^1(k_{v_1},C) \arrow{r} & 1 \arrow{r} &1. \label{split}
\end{tikzcd}
\end{equation}
The first row splits, and thus we may identify $\rH^1(k,C) \cong \{\pm 1\} \oplus l_0 / l^{\times n}$; then $l_0 / l^{\times n}$ is precisely the kernel of the canonical map $\rH^1(k,C) \to \rH^1(k_{v_1},C)$. Now since the adjoint map $G(k_{v_1}) \to \overline{G}(k_{v_1})$ is surjective (recall that $G(k_{v_1}) \cong \SU_n(\R)$), we have in $(\mathrm{L}_{v_1})$ that the image of $\overline{G}(k)$ in $\rH^1(k_{v_1},C)$ is trivial, hence $\delta(\overline{G}(k)) \subset l_0 / l^{\times n}$. 

If $n$ is even and $V_\infty = \{v_0\}$, then $k = \Q$. We have, for each $v \in V_f$,
\begin{equation}
\begin{tikzcd}
1 \arrow{r} & \frac{\mu(k)}{N(\mu(l))} \arrow{r}\arrow{d} & \rH^1(k,C) \arrow{r}\arrow{d} & l_0 / l^{\times n} \arrow{r}\arrow{d} & 1 \\
1 \arrow{r} & \frac{\mu(k_v)}{N(\mu_n(l \ot k_v))} \arrow{r} & \rH^1(k_{v},C) \arrow{r} & \frac{\ker(N: l \ot k_v \to k_v)}{(l \ot k_v)^{\times n}} \arrow{r} &1. \notag
\end{tikzcd}
\end{equation}
We observe that $\mu(k) / N(\mu(l))$ $(\cong \{\pm 1\})$ acts trivially on $\Delta_v$ for every $v \in V_f$ (see for example \cite[§4]{MohammadiSalehi12}), hence the action factors through $l_0 / l^{\times n}$. Thus $\# H^1(k,C)_\xi = 2 \cdot \# l_\xi / l^{\times n}$, where $l_\xi/ l^{\times n} = \{ x \in l_0/ l^{\times n} \mid \xi(x) = 1\}$, so that the bound we establish below will hold with an extra factor $\tilde{n}$ in the case $k = \Q$. \bigbreak

It remains to understand the action of $l_0 / l^{\times n}$ on $\Delta$. Let $x \in l$ and let
$$(x) = \prod_{\fP} \fP^{i_\fP} \overline{\fP}{}^{i_{\overline{\fP}}} \cdot \prod_{\fp'} \fp'^{i_{\fp'}} \cdot \prod_{\fP''} \fP''^{i_{\fP''}}$$
be the unique factorization of the fractional ideal of $l$ generated by $x$, where $(\fP, \overline{\fP})$ (resp.~$\fp'$, $\fP''$) runs over the set of primes of $l$ that lie over primes of $k$ that split over $l$ (resp.~over inert primes of $k$, over ramified primes of $k$). When $x \in l_0$, $N(x) \in k^{\times n}$ and thus $n$ divides $i_\fP + i_{\overline{\fP}}$, $2i_{\fp'}$ and $i_{\fP''}$. 

Observe that $v \in V_f$ splits over $l$ if and only if $l$ embeds into $k_v$, that is, if and only if ($\cG$ splits over $k_v$ and) $G$ is an inner $k_v$-form of $\SL_n$. In particular, at such a place $v$, $G$ is isomorphic to $\SL_{n_v} \fD_v$ for some central division algebra $\fD_v$ over $k_v$ of index $d_v = n / n_v$. 
In \cite[§4]{MohammadiSalehi12}, it is shown that when $v$ splits as $\fP \overline{\fP}$ over $l$, the action of $x \in l_0$ is analogous to the inner case described in \ref{indexinner}, hence $x$ acts trivially on $\Delta_v$ if and only if $n$ divides $d_v i_\fP$ (and thus $n$ also divides $d_v i_{\overline{\fP}}$), i.e. $v_{\fP} (x) = 0 \mod n_v$ (and $v_{\overline{\fP}}(x) = 0 \mod n_v$). 
When $v$ is inert, say $v$ corresponds to $\fp'$, then $x$ acts trivially on $\Delta_v$ if and only if $n$ divides $i_{\fp'}$ \cite[§4]{MohammadiSalehi12}. 

Let $T$ be the set of places $v \in V_f$ such that $v$ splits over $l$ and $G$ is not split over $k_v$, and let $T^l$ be a subset of the finite places of $l$ consisting of precisely one extension of each $v \in T$, so that restriction to $k$ defines a bijection from $T^l$ to $T$. By the discussion above, we can form an exact sequence
$$1 \to (l_n \cap l_0) / l^{\times n} \to l_\xi / l^{\times n} \xrightarrow{(w)_{w \in T^l}} \bigoplus_{w \in T^{l}} \Z / n\Z,$$
where $l_n = \{ x \in l^\times \mid w(x) \in n\Z \textrm{ for each normalized finite place $w$ of $l$}\}$ and $l_\xi / l^{\times n} = \{ x \in l_0/ l^{\times n} \mid \xi(x) = 1\}$. Moreover, the image of $l_\xi / l^{\times n}$ lies in the subgroup $\bigoplus_{w \in T^l} n_v \Z / n \Z$. Thus, if we assume $k \neq \Q$ (so that we may identify $\delta(\ol{G}(k))$ with a subgroup of $l_0 / l^{\times n}$), 
$$\# \left( \delta(\ol{G}(k)) \cap \rH^1(k,C)_\xi \right) \leq \# \left(l_\xi / l^{\times n} \right) \leq \# \left( (l_n \cap l_0) / l^{\times n} \right) \cdot \prod_{v \in T} d_v.$$
We get the concrete bound on the index 
$$[\Gamma: \Lambda] \leq h_l \tilde{n}^{m} n \cdot \prod_{v \in T} d_v \cdot \prod_{v \in V_f} \# \Xi_{\Theta_v}$$
by combining this with \eqref{index} and lemma \ref{dirichletunits}. If $k = \Q$, we have instead
$$[\Gamma: \Lambda] \leq h_l \tilde{n}^{m+1} n \cdot \prod_{v \in T} d_v \cdot \prod_{v \in V_f} \# \Xi_{\Theta_v}. $$

%%%%%%%%%%

\section{The field $k$ is $\Q$} \label{kisQ}

We set $m= [k:\Q]$ and as before, $n=r+1$. The purpose of this section is to show that $k = \Q$, i.e.~$m = 1$. \bigbreak

We start by recalling that if $P_v$ is special (in particular, if it is hyperspecial), i.e.~$\Theta_v$ consists of a single special (resp.~hyperspecial) vertex of $\Delta_v$, then $\Xi_{\Theta_v}$ is trivial. Regardless of the type $\Theta_v$, we have $\# \Xi_{\Theta_v} \leq \tilde{n}$ unless $G$ is an inner $k_v$-form of $\SL_n$ (say $G \cong SL_{n_v} (\fD_v)$), in which case $\# \Xi_{\Theta_v} \leq \# \Delta_v = n_v$, where $n_v-1$ is the rank of $G$ over $k_v$. (For example, this can be seen explicitly on all the possible relative local Dynkin diagrams $\Delta_v$ for $G(k_v)$, enumerated in \cite[§4]{Tits79} or \cite[§2]{MohammadiSalehi12}. In the inner case, the Dynkin diagram is a cycle on which the adjoint group acts as rotations.) \bigbreak

By a result of Kneser \cite{Kneser65}, $G$ is quasi-split over the maximal unramified extension $\widehat{k}_v$ of $k_v$ for any $v \in V_f$. This means that over $\widehat{k}_v$, $G$ is isomorphic to $\cG$. The quasi-split $k$-forms of simply connected absolutely almost simple groups of type $A_{n-1}$ are well understood \cite{Tits66}: either $\cG \cong \SL_n$, or $\cG \cong \SU_{n,l}$, the special unitary group associated to the split hermitian form on $l^n$, where $l$ is a quadratic extension of $k$ equipped with the canonical involution (incidentally, $l$ is the splitting field of $\SU_{n,l}$, in accordance with the notation introduced). Thus, over $\widehat{k}_v$, only these two possibilities arise for $G$. (Nonetheless, $\cG$ might split over $\widehat{k}_v$; in fact, it does so except at finitely many places.) In particular, the rank $r_v$ of $G$ over $\widehat{k}_v$ is either $r$, or the ceiling of $r/2$. 

%%%%%

\subsection{The inner case} \label{inner}

The case where $G$ is an inner $k$-form of $\SL_n$ (i.e.~when $l=k$) has been treated in section \ref{setup}. We observed that if $G$ is an inner $k_v$-form of $\SL_n$ for some $v \in V_\infty$, then $G(k_v)$ cannot be compact. This forced $V_\infty = \{v_0\}$ and thus $k =  \Q$. 

%%%%%

\subsection{The outer case} \label{outer}

Here we settle the case where $G$ is an outer $k$-form of $\SL_n$, i.e.~when $[l:k] = 2$. We observed in section \ref{setup} that $l$ has two real embeddings (extending $k \to k_{v_0}$) and $m-1$ pairs of conjugate complex embeddings. Suppose that $m > 1$. \bigbreak

Let $T$ be the finite set of places $v \in V_f$ such that $v$ splits over $l$ and $G$ is not split over $k_v$. Then, according to section \ref{indexouter}, we have
$$[\Gamma: \Lambda] \leq h_l \tilde{n}^{m} n \cdot \prod_{v \in T} d_v \cdot \prod_{v \in V_f} \# \Xi_{\Theta_v}$$
where $\tilde{n} = 1$ or 2 if $n$ is odd or even, and $h_l$ denotes the class number of $l$. Combined with \eqref{volume}, we find (abbreviating $V_n = \prod_{i=1}^{n-1} \frac{i!}{(2\pi)^{i+1}}$)
$$\mu(G(k_{v_0})/ \Gamma) \geq \tilde{n}^{-m} n^{-1} h_l^{-1} D_k^{\frac{n^2 -1}{2}} (D_l / D_k^{2})^{\frac{1}{2}\fs(\cG)} V_n^{m} \cdot \prod_{v \in T} d_v^{-1} \cdot \prod_{v \in V_f} \# \Xi_{\Theta_v}^{-1} \cdot \prod_{v \in V_f} e(P_v).$$

We use \cite[prop.~2.10, rem.~2.11]{Prasad89} and the observations made at the begining of section \ref{kisQ} to study the local factors of the right-hand side. 
\begin{enumerate}[label=(\roman*),itemsep=0pt, topsep=4pt]
\item If $v \in T$, then we use $e(P_v) \geq (q_v-1) q_v^{(n^2 -n^2 d_v^{-1} - 2)/2}$ to obtain $d_v^{-1} \cdot \# \Xi_{\Theta_v}^{-1} \cdot e(P_v) \geq n^{-1} \cdot (q_v-1) q_v^{n^2 /4 -1} > 1$ when $n \geq 4$. When $n = 3$, then $d_v = 3$ and we also have $d_v^{-1} \cdot \# \Xi_{\Theta_v}^{-1} \cdot e(P_v) \geq n^{-1} \cdot (q_v-1) q_v^{n^2 /3 -1} > 1$ (lemma \ref{ePv}).

\item If $v \notin T$ but $P_v$ is special, then $\# \Xi_{\Theta_v} = 1$ and $e(P_v) > 1$, thus $\# \Xi_{\Theta_v}^{-1} \cdot e(P_v) > 1$.

\item If $v \notin T$, $P_v$ is not special and $G$ is not split over $k_v$, then we use that $e(P_v) \geq (q_v+1)^{-1} q_v^{r_v+1}$ to obtain $\# \Xi_{\Theta_v}^{-1} \cdot e(P_v) \geq \tilde{n}^{-1} \cdot (q_v+1)^{-1} q_v^{\ceil{(n-1)/2}+1} > 1$ (lemma \ref{ePv2}). 

\item If $v \notin T$, $P_v$ is not special but $G$ splits over $k_v$, then $P_v$ is properly contained in a hyperspecial parahoric $H_v$. There is a canonical surjection $H_v \to \SL_n(\ff_v)$, under which the image of $P_v$ is the proper parabolic subgroup $\overline{P}_v$ of $\SL_n(\ff_v)$ whose type consists of the vertices belonging to the type of $P_v$ in the Dynkin diagram obtained by removing the vertex corresponding to $H_v$ in the affine Dynkin diagram of $G(k_v)$. In particular, it follows that $[H_v: P_v] = [\SL_n(\ff_v) : \overline{P}_v]$ and we may compute using lemma \ref{indexPv}
$$e(P_v) = [H_v: P_v] \cdot e(H_v) > [H_v: P_v] > q^{n-1}.$$
Hence $\# \Xi_{\Theta_v}^{-1} \cdot e(P_v) >  n^{-1} q^{n-1} > 1$.
\end{enumerate}
Multiplying all the factors together, we have that 
$$\prod_{v \in T} d_v^{-1} \cdot \prod_{v \in V_f} \# \Xi_{\Theta_v}^{-1} \cdot \prod_{v \in V_f} e(P_v) > 1$$
and we can thus write 
\begin{equation}
\mu(G(k_{v_0})/ \Gamma) > \tilde{n}^{-m} n^{-1} h_l^{-1} D_k^{\frac{n^2 -1}{2}} (D_l / D_k^{2})^{\frac{1}{2}\fs(\cG)} V_n^{m}. \label{Bout}
\end{equation}

Recall that $D_l / D_k^{2}$ is the norm of the relative discriminant $\mathfrak{d}_{l/k}$ of $l$ over $k$; in particular, $D_l / D_k^{2}$ is a positive integer. Note also that $\fs(\cG) \geq 5$ if $n \geq 3$. 
We combine this with two number-theoretical bounds: from the results in \cite[§6]{BorelPrasad89}, we use that 
$$h_l^{-1}D_l \geq \frac{1}{100} \left( \frac{12}{\pi} \right)^{2m};$$ 
from Minkowski's geometry of numbers, we recall ($k$ is totally real)
$$D_k^{\frac{1}{2}} \geq \frac{m^m}{m!}.$$
Altogether, we obtain
\begin{align}
\mu(G(k_{v_0})/ \Gamma) &> \frac{1}{100 \tilde{n}^{m}} \left( \frac{12}{\pi} \right)^{2m} D_k^{\frac{n^2 -5}{2}} (D_l / D_k^{2})^{\frac{1}{2}\fs(\cG) -1} V_n^{m} n^{-1} \label{Bout2} \\
 					&\geq \frac{1}{100\tilde{n}^{m}} \left( \frac{12}{\pi} \right)^{2m} \left( \frac{m^m}{m!} \right)^{n^2 -5} V_n^{m} n^{-1}. \notag
\end{align}
We consider the function $M: \N \times \N \rightarrow \R$ defined by
$$M(m,n) = \frac{1}{100 \tilde{n}^{m}} \left( \frac{12}{\pi} \right)^{2m} \left( \frac{m^m}{m!} \right)^{n^2 -5} \left( \prod_{i=1}^{n-1} \frac{i!}{(2\pi)^{i+1}} \right)^{m-1} n^{-1}.$$
$M$ is strictly increasing in both variables, provided $m \geq 2$ and $n \geq 6$ (lemma \ref{Mout}). In consequence, if $m \geq 2$, $n \geq 9$, 
$$\frac{\mu(G(k_{v_0})/ \Gamma)}{\mu(\SL_n(\R)/ \SL_n(\Z))} > \frac{M(m,n)}{\prod_{i=2}^{n} \zeta(i)} > \frac{M(2,9)}{\prod_{i=2}^{\infty} \zeta(i)} >1,$$
cf.~lemma \ref{prodzeta}, and $\Gamma$ is not of minimal covolume. \bigbreak

In a similar manner, we would like to show that $m$ cannot be large. To this end, Odlyzko's bounds on discriminants \cite[table 4]{Odlyzko76} are well-suited. We have 
$$D_k^{\frac{1}{2}} > A^m \cdot E, \textrm{ with $A = 29.534^{\frac{1}{2}}$ and $E = e^{-4.13335}$.}$$
Combining with \eqref{Bout2}, we obtain
\begin{align}
\mu(G(k_{v_0})/ \Gamma) &> \frac{1}{100\tilde{n}^{m}} \left( \frac{12}{\pi} \right)^{2m} \left( A^m E \right)^{n^2 -5} V_n^{m} n^{-1}. \notag
\end{align}
We consider the function $M': \N \times \N \rightarrow \R$ defined by 
$$M'(m,n) = \frac{1}{100 \tilde{n}^{m}} \left( \frac{12}{\pi} \right)^{2m} \left( A^m E \right)^{n^2 -5} \left( \prod_{i=1}^{n-1} \frac{i!}{(2\pi)^{i+1}} \right)^{m-1} n^{-1}.$$
$M'$ is also strictly increasing in both variables, provided $m \geq 4$ and $n \geq 4$ (lemma \ref{Mout2}). This means that if $m \geq 6$, $n \geq 4$, 
$$\frac{\mu(G(k_{v_0})/ \Gamma)}{\mu(\SL_n(\R)/ \SL_n(\Z))} > \frac{M'(m,n)}{\prod_{i=2}^{n} \zeta(i)} > \frac{M'(6,4)}{\prod_{i=2}^{\infty} \zeta(i)} >1,$$
(cf.~table \ref{tableMout2} and lemma \ref{prodzeta}) and $\Gamma$ is not of minimal covolume. \bigbreak

We may thus restrict our attention to the range $4 \leq n \leq 8$ and $2 \leq m \leq 5$ (we will treat the case $n=3$ with a separate argument at the end of this section). By further sharpening our estimates on the discriminant, we will show that all these values are excluded as well, forcing $m=1$. 

From the bound \eqref{Bout2} and the estimate $\mu(G(k_{v_0})/ \Gamma) \leq \mu(\SL_n(\R)/ \SL_n(\Z)) < 2.3 \cdot V_n$ (\ref{prodzeta}), we deduce an upper bound on the discriminant of $k$:
\begin{align}
D_k &< \left( 230 \tilde{n}^{m} \left( \frac{\pi}{12} \right)^{2m} (D_l / D_k^{2})^{1 - \frac{1}{2}\fs(\cG)} V_n^{1-m} n \right)^{\frac{2}{n^2-5}} \label{Bout3} \\
& \leq \left( 230 \tilde{n}^{m} \left( \frac{\pi}{12} \right)^{2m} V_n^{1-m} n \right)^{\frac{2}{n^2-5}} =: C(m,n). \notag
\end{align}
As can be seen by comparing the values of $C$ (table \ref{tableCout}) with the smallest discriminants (table \ref{tabledisck}), this bound already rules out $n \geq 7$. We use these two tables to obtain information about $D_k$. A lower bound on $D_k$ in turn will give us a bound on the relative discriminant: using \eqref{Bout2} again, 
\begin{equation}
D_l / D_k^2 < \left( 230 \tilde{n}^{m} \left( \frac{\pi}{12} \right)^{2m} D_k^{\frac{5-n^2}{2}} V_n^{1-m} n \right)^{\frac{2}{\fs(\cG) -2}}. \label{Brel}
\end{equation}
We proceed to rule out all values of $m$. In what follows, unless specified otherwise, any bound on $D_k$ is obtained using \eqref{Bout3}, (\ref{tableCout}) or (\ref{tabledisck}), and any upper bound on $D_l / D_k^2$ using \eqref{Brel}. Claims made on the existence of a field $l$ satisfying certain conditions are always made with the underlying assumption that $l$ is a quadratic extension of $k$ of signature $(2,m-1)$.  

\begin{enumerate}[label=(\roman*),itemsep=0pt]

\item[$\underline{m=5}$] gives $14641 \leq D_k \leq 15627$ (and $n=4$). A quick look in the online database of number fields \cite{JonesRoberts14} shows%
\footnote{The database \cite{JonesRoberts14} provides a certificate of completeness for certain queries. All allusions made here refer to searches that are proven complete. 
However, it is important to note that in \cite{JonesRoberts14}, class numbers are computed assuming the generalized Riemann hypothesis (the rest of the data being unconditional). The class numbers referred to in this paper were therefore all verified using PARI/GP's \texttt{bnfcertify} command. A PARI/GP script of this process is available on the author's page (\href{http://www.math.ucsd.edu/~fthilman/research/mincovsln/classnumberscertificate.html}{\texttt{math.ucsd.edu/{\textasciitilde}fthilman/}}).}
 that there is only one such field (with $D_k = 14641$). Now for $l$, Odlyzko's bound \cite[table 4]{Odlyzko76} reads 
$$D_l > (29.534)^2 \cdot (14.616)^8 \cdot e^{-8.2667} \geq 4.66756 \cdot 10^8$$
and in particular, we compute that $D_l/ D_k^2 \geq 2.177$ (hence $D_l / D_k^2 \geq 3$). On the other hand, \eqref{Brel} yields 
$$D_l / D_k^2 < 1.271,$$
ruling out this case. 

\item[$\underline{m=4}$] gives $725 \leq D_k \leq 1741$ (and $n = 4$). A quick look in the database \cite{JonesRoberts14} shows that there are three fields satisfying this requirement, with discriminants respectively 725, 1125, 1600.
\begin{enumerate}[label=(\roman*),itemsep=0pt]
\item If $D_k = 1600$, then $D_l / D_k^2 < 1.365$, hence $D_l = D_k^2 = 2560000$. But, as observed in the database, there are no fields $l$ of signature $(2,3)$ with $D_l \leq 3950000$. 
\end{enumerate}
Unfortunately, the database has no complete records for fields with signature $(2,3)$ and discriminants past 3950000. We will thus need to refine our bounds to be able to treat the two other possible values for $D_k$. First, we go back to our bound on the class number $h_l$: as in \cite[§6]{BorelPrasad89}, we use Zimmert's bound $R_l \geq 0.04 \cdot e^{2 \cdot 0.46 + (m-1) \cdot 0.1}$ on the regulator of $l$ along with the Brauer-Siegel theorem (with $s=2$) to deduce
$$h_l \leq 100 \cdot e^{- 0.82 - 0.1 \cdot m} \cdot (2\pi)^{-2m} \cdot \zeta(2)^{2m} \cdot  D_l  \leq 29.523 \cdot \left( \frac{\pi}{12} \right)^8 \cdot D_l.$$
Using this, we may rewrite the bound \eqref{Brel} as 
\begin{equation*}
D_l / D_k^2 < \left( 67.9029 \tilde{n}^{4} \left( \frac{\pi}{12} \right)^{8} D_k^{\frac{5-n^2}{2}} V_n^{-3} n \right)^{\frac{2}{\fs(\cG) -2}}. 
\end{equation*}
\begin{enumerate}[label=(\roman*),itemsep=0pt, start=2]
\item If $D_k = 1125$, then our new bound yields $D_l / D_k^2 \leq 2$, hence $D_l \leq 2 D_k^2 = 2531250$ and this is ruled out by the database. 
\item If $D_k = 725$, then our new bound yields $D_l / D_k^2 \leq 11$, hence $D_l \leq 11 D_k^2 = 5781875$. Selmane \cite{Selmane99} has computed all fields of signature $(2,3)$ that possess a proper subfield and have discriminant $D_l \leq 6688609$. It turns out that among those, only the field with discriminant $-5781875$ can be an extension of $k$. As observed in the online database, this field has class number 1. Substituting this information in \eqref{Bout}, we see that the right-hand side exceeds $2.3 \cdot V_n$. 
\end{enumerate}

\item[$\underline{m=3}$] gives $49 \leq D_k \leq 194$ (and $n = 4$ or $5$). A quick look in the database \cite{JonesRoberts14} shows that there are four fields satisfying this requirement, with discriminants respectively 49, 81, 148, 169. 
\begin{enumerate}[label=(\roman*),itemsep=0pt]
\item If $D_k = 169$, then $D_l / D_k^2 < 1.661$ hence $D_l = D_k^2 = 28561$. There are no fields $l$ with $D_l \leq 28000$. 
\item If $D_k = 148$, then $D_l / D_k^2 \leq 2$. There are no fields $l$ with $D_l / 148^2 = 1$ or 2. 
\item If $D_k = 81$, then $D_l / D_k^2 \leq 24$. An extensive search in the database shows that this can only be satisfied by one field $l$, with discriminant $D_l = 81^2 \cdot 17$. It has class number $h_l = 1$, hence we may substitute this information in \eqref{Bout} and compute that the right-hand side exceeds $2.3 \cdot V_n$. 
\item If $D_k = 49$, then $D_l / D_k^2 \leq 155$. An extensive search in the database shows that there are 6 fields $l$ satisfying this condition. They correspond to $D_l / D_k^2 = 13, 29, 41, 64, 97$ or $113$, and all have class number 1. Then, in \eqref{Bout}, the right-hand side again exceeds $2.3 \cdot V_n$ (note that it suffices to check this for the smallest value of $D_l/D_k^2$). 
\end{enumerate}

\item[$\underline{m=2}$] gives $5 \leq D_k \leq 21$ (and $4 \leq n \leq 6$). It is well known (and can be observed in the database \cite{JonesRoberts14}) that there are 6 fields satisfying this requirement, with discriminants respectively 5, 8, 12, 13, 17, 21. From \eqref{Brel}, we see that $D_l/ D_k^2 \leq$ 214, 38, 8, 6, 2, 1 respectively. 
\begin{enumerate}[label=(\roman*),itemsep=0pt]
\item If $D_k = 21$ or 17, we observe that $D_l \leq 578$. There are no fields with $D_l \leq 578$ that can be extensions of $k$ in these cases. 
\item If $D_k = 13$, then the database exhibits only one possible field $l$ with $D_l = 13^2 \cdot 3$. This field has trivial class group, and using this information in \eqref{Bout}, we see that the right-hand side exceeds $2.3 \cdot V_n$. 
\item If $D_k = 12$, then there are again no fields with $D_l \leq 8 D_k^2$. 
\item If $D_k = 8$, then there are 11 candidates $l$ with $D_l \leq 38 \cdot 8^2$, and all have trivial class group. The one with smallest relative discriminant has $D_l / D_k^2 = 7$. For this field (hence for all of them), the right-hand side of \eqref{Bout} is again too large. 
\item If $D_k = 5$, there are 25 candidates $l$ with $D_l \leq 214 \cdot 5^2$, and all have trivial class group. The one with smallest relative discriminant has $D_l = 11$. This field (hence all of them) is one more time excluded by \eqref{Bout}. 
\end{enumerate}
\end{enumerate}
\bigbreak \smallskip

It remains to deal with the case {$n=3$}. First, we proceed as above, using lemma \ref{Mout2}, $M'(16,3) \simeq 4.6751...$, and $\zeta(2) \cdot \zeta(3) < 1.97731$ to see that 
$$\frac{\mu(G(k_{v_0})/ \Gamma)}{\mu(\SL_3(\R)/ \SL_3(\Z))} > \frac{M'(m,3)}{\zeta(2) \cdot \zeta(3)} >1$$
provided $m \geq 16$. Hence we may restrict our attention to the range $2 \leq m \leq 15$. 

Unfortunately, this bound on the degree of $k$ is too large to allow us to work with a number field database. Of course, the reason this bound is large is that the powers of $D_k$ and $D_l$ appearing in \eqref{Bout} are very small. In turn, the bound we used for the class number $h_l$ was very greedy in terms of $D_l$, aggravating the situation. In fact,  we can use \eqref{Bout} and one of Odlyzko's bounds \cite{Odlyzko76} for $D_l$ to obtain a lower bound on $h_l$:
\begin{equation}
h_l \geq \frac{D_k^{-1} D_l^{\frac{5}{2}} V_3^{m-1}}{3 \cdot \zeta(2) \cdot \zeta(3)} \geq \frac{D_l^2 V_3^{m-1}}{3 \cdot \zeta(2) \cdot \zeta(3)} > \frac{(25.465^2 \cdot 13.316^{2m-2} \cdot e^{-7.0667})^2 \cdot V_3^{m-1}}{3 \cdot \zeta(2) \cdot \zeta(3)}. \label{Hout}
\end{equation}
We record the values of this bound in table \ref{tablehl} (for small values of $m$, we used the actual minimum for $D_l$ to obtain this lower bound for $h_l$).

To solve this issue, we use the following trick. The Hilbert class field $L$ of $l$ has degree $[L: \Q] = 2 m h_l$, signature $(2h_l, (m-1) h_l)$ and discriminant $D_L = D_l^{h_l}$. Hence, when the class number is large, we can use Odlyzko's bounds \cite{Odlyzko76} for $D_L$ in order to improve our bounds on $D_l$. Namely, we have
$$D_l = D_L^{\frac{1}{h_l}} > 60.015^{2} \cdot 22.210^{2m-2} \cdot e^{\frac{-80.001}{h_l}}.$$
We record this bound for $D_l$ in table \ref{tableDl}. 

Now using $D_l \geq D_k^2$, we may rewrite \eqref{Bout2} as
$$\zeta(2) \cdot \zeta(3) \cdot V_3 > \mu(G(k_{v_0})/ \Gamma) > \frac{1}{300} \left( \frac{12}{\pi} \right)^{2m} D_l \cdot V_3^{m}$$
and check that this inequality contradicts the bound in table \ref{tableDl} as soon as $m \geq 4$. For $m=3$ and $m=2$, the bound reads respectively $D_l \leq 4578732$ and $D_l \leq 13643$. 

Finally, to treat the remaining two cases, we can use the online database \cite{JonesRoberts14}. If $m=3$, we observe that all fields of signature $(2,2)$ with discriminant $D_l \leq 4578732$ have class number either $h_l = 1$ or $h_l = 2$; this contradicts \eqref{Hout} and table \ref{tablehl}. Similarly, if $m=2$, we observe in the database that all fields of signature $(2,1)$ with discriminant $D_l \leq 13643$ also have class number either $h_l = 1$ or $h_l = 2$. This is again a contradiction to \eqref{Hout} and table \ref{tablehl}. \bigbreak

\begin{rem*}
Below is a summary of the various discriminant bounds that were used in this section to exclude a given couple $(m,n)$ from giving rise to a lattice of minimal covolume. 
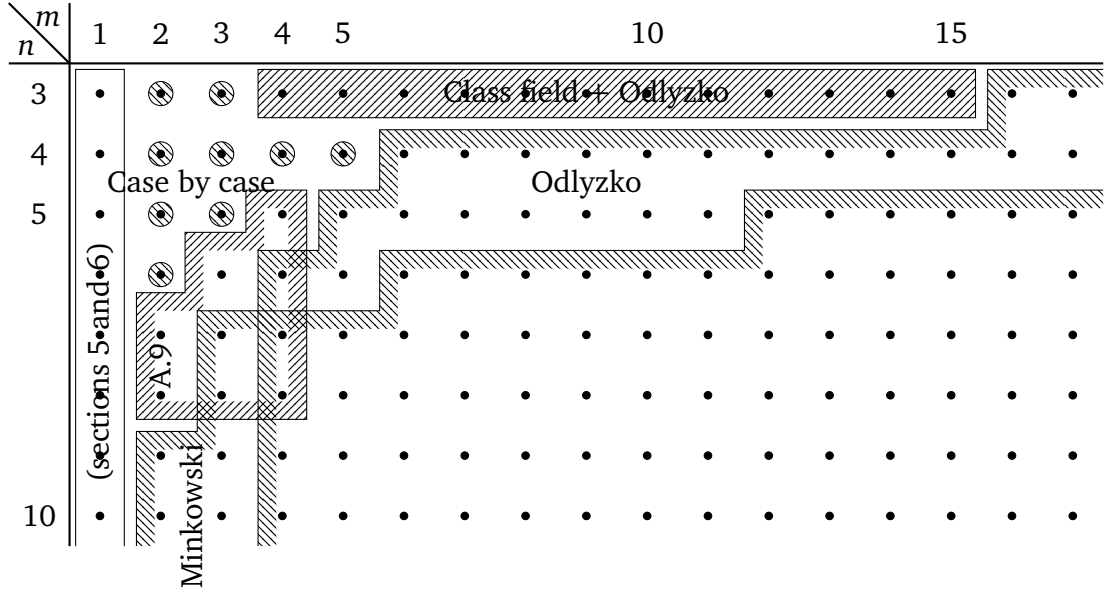
\begin{center} \begin{tikzpicture}[scale=0.8]
    \coordinate (Origin)   at (0,0);
    \coordinate (XAxisMin) at (-1,0);
    \coordinate (XAxisMax) at (17,0);
    \coordinate (YAxisMin) at (0,-8);
    \coordinate (YAxisMax) at (0,1);
    \draw [thick,black] (XAxisMin) -- (XAxisMax);
    \draw [thick,black] (YAxisMin) -- (YAxisMax);
    \draw [thick,black] (-1,1) -- (0,0);
    \node at (0,1) [below left,inner sep=3pt] {$m$};
    \node at (-1,0) [above right,inner sep=3pt] {$n$};
    
    \node at (-0.5,-0.5) {3};
    \node at (-0.5,-1.5) {4};
    \node at (-0.5,-2.5) {5};
    \node at (-0.5,-7.5) {10};

    \node at (0.5,0.5) {1};
    \node at (1.5,0.5) {2};
    \node at (2.5,0.5) {3};
    \node at (3.5,0.5) {4};
    \node at (4.5,0.5) {5};
    \node at (9.5,0.5) {10};
    \node at (14.5,0.5) {15};

\begin{scope}
\draw[pattern = north east lines]
(1.1,-5.9) -- (1.1,-3.8) -- (1.9,-3.8) -- (1.9,-2.8) -- (2.9,-2.8) --(2.9,-2.1) -- (3.9,-2.1) -- (3.9,-5.9) -- cycle;
\fill[white] (1.4,-5.6) -- (1.4,-4.1) -- (2.2,-4.1) -- (2.2,-3.1) -- (3.2,-3.1) --(3.2,-2.4) -- (3.6,-2.4) -- (3.6,-5.6) -- cycle;
\end{scope}
\node[rotate=90] at (1.5,-5) {\ref{tabledisck}};

    \foreach \x in {1,2,...,17}{
      \foreach \y in {-1,-2,...,-8}{ 
        \node[draw,circle,inner sep=1pt,fill] at (\x-0.5,\y+0.5) {};
       }}
    
\begin{scope}
\clip (1,-8) -- (1,-6) -- (2,-6) -- (2,-4) -- (5,-4) -- (5,-3) -- (11,-3) -- (11,-2) -- (17,-2) -- (17,-2.4) -- (11.4,-2.4) -- (11.4, -3.4) -- (5.4,-3.4) -- (5.4,-4.4) -- (2.4, -4.4) -- (2.4,-6.4) -- (1.4, -6.4) -- (1.4, -8) -- cycle;
\draw[pattern=north west lines]
	(1.1,-10) -- (1.1,-6.1) -- (2.1,-6.1) -- (2.1,-4.1) -- (5.1,-4.1) -- (5.1,-3.1) -- (11.1,-3.1) -- (11.1,-2.1) -- (17.1,-2.1)--(17.1,-3.1);
\end{scope}
\node[rotate=90] at (2,-7.5) {Minkowski};

\begin{scope}
\clip 	(3,-8) -- (3,-3) -- (4,-3) -- (4,-2) -- (5,-2) -- (5,-1) -- (14,-1) -- (14,0) -- (17,0) -- (17,-0.4) -- (15.4, -0.4) -- (15.4,-1.4) -- (5.4,-1.4) -- (5.4,-2.4) -- (4.4,-2.4) -- (4.4,-3.4) -- (3.4,-3.4) -- (3.4,-8) -- cycle; 
\draw[pattern = north west lines]
	(3.1,-10) -- (3.1,-3.1) -- (4.1,-3.1) -- (4.1,-2.1) -- (5.1,-2.1) -- (5.1,-1.1) -- (15.1,-1.1) -- (15.1,-0.1) -- (17.1,-0.1) -- (17.1,-2.1); 
\end{scope}
\node at (8.5,-2) {Odlyzko};

\draw
(0.1,-8) -- (0.1,-0.1) -- (0.9,-0.1) -- (0.9,-8);
\node[rotate=90] at (0.5,-4.95) {(sections \ref{Ginner} and \ref{hyperspecial})};

\draw[pattern = north west lines]
(1.5,-3.5) circle (0.2)
(1.5,-2.5) circle (0.2)
(1.5,-1.5) circle (0.2)
(1.5,-0.5) circle (0.2)
(2.5,-2.5) circle (0.2)
(2.5,-1.5) circle (0.2)
(2.5,-0.5) circle (0.2)
(3.5,-1.5) circle (0.2)
(4.5,-1.5) circle (0.2);
\node at (2,-2) {Case by case};

\begin{scope}
\draw[pattern=north east lines]
	(3.1,-0.9) -- (3.1,-0.1) -- (14.9,-0.1) -- (14.9,-0.9) -- cycle;
\end{scope}
\node at (8.5,-0.5) {Class field + Odlyzko};

\end{tikzpicture} \end{center}
\end{rem*}

%%%%%%%%%%

\section{$G$ is an inner form of $\SL_n$} \label{Ginner}

The purpose of this section is to show that $G$ is an inner $k$-form of $\SL_n$, i.e.~that $\cG$ splits over $k$. Let us thus suppose, for contradiction, that $[l : k] > 1$. \bigbreak

We have shown in section \ref{kisQ} that $k= \Q$, so that the bounds \eqref{Bout} and \eqref{Bout2} obtained in \ref{outer} can be adapted as follows: (the extra factor $\tilde{n}$ is due to the correction in the index bound when $k=\Q$, cf.~section \ref{indexouter})
\begin{align*}
\mu(G(k_{v_0})/ \Gamma) &> \tilde{n}^{-2} n^{-1} h_l^{-1} D_l^{\frac{1}{2}\fs(\cG)} V_n \\
					&\geq \frac{1}{100 \tilde{n}^2} \left( \frac{12}{\pi} \right)^{2} D_l^{\frac{1}{2}\fs(\cG) -1} V_n n^{-1}.
\end{align*}

First, let us assume that $h_l \neq 1$. Since $l$ is totally real, this implies $D_l \geq 40$. Note that $\fs(\cG) \geq \frac{1}{2}(r^2+r-2) = \frac{1}{2}(n^2 - n -2)$. Therefore
$$\mu(G(k_{v_0})/ \Gamma) > \frac{1}{100\tilde{n}^2} \left( \frac{12}{\pi} \right)^{2} 40^{\frac{1}{4}(n^2-n-6)} V_n n^{-1}.$$

We consider the function $N: \N \rightarrow \R$ defined by
$$N(n) = \frac{1}{100 \tilde{n}^2} \left( \frac{12}{\pi} \right)^{2} 40^{\frac{1}{4}(n^2-n-6)} n^{-1}.$$
$N$ is strictly increasing, provided $n \geq 2$ (lemma \ref{Nn}). In consequence, if $n \geq 4$, then $N(n) \geq N(4) \simeq 2.30692...$ and thus 
$$\frac{\mu(G(k_{v_0})/ \Gamma)}{\mu(\SL_n(\R)/ \SL_n(\Z))} > \frac{N(n)}{\prod_{i=2}^{n} \zeta(i)} > \frac{N(4)}{\prod_{i=2}^{\infty} \zeta(i)} >1,$$
hence $\Gamma$ is not of minimal covolume. For $n=3$ we notice that $\fs(\cG) = 5$, so that
$$\mu(G(k_{v_0})/ \Gamma) > \frac{1}{300} \left( \frac{12}{\pi} \right)^{2} 40^{\frac{3}{2}} \cdot V_3> 12.3035 \cdot V_3$$
and $\Gamma$ is not of minimal covolume. \bigbreak

Second, if $h_l = 1$, then at least $D_l \geq 5$ and we may consider the function $N': \N \to \R$ defined by 
$$N'(n) = \tilde{n}^{-2} n^{-1} 5^{\frac{1}{4}(n^2 - n - 2)}.$$
$N'$ is strictly increasing (lemma \ref{Nn}) and $N'(4) \simeq 3.49385...$, thus
$$\frac{\mu(G(k_{v_0})/ \Gamma)}{\mu(\SL_n(\R)/ \SL_n(\Z))} > \frac{N(n)}{\prod_{i=2}^{n} \zeta(i)} > \frac{N(4)}{\prod_{i=2}^{\infty} \zeta(i)} >1,$$
and $\Gamma$ is not of minimal covolume. For $n=3$, we use again that $\fs(\cG) = 5$ to see that
$$\mu(G(k_{v_0})/ \Gamma) > \frac{1}{3} \cdot 5^{\frac{5}{2}} \cdot V_3> 18.6338 \cdot V_3$$
and $\Gamma$ is not of minimal covolume. This forces $l = k$ and $G$ to be an inner form.

%%%%%%%%%%

\section{The parahorics $P_v$ are hyperspecial and $G$ splits at all places} \label{hyperspecial}

So far, we have established that $k = l = \Q$ and $G$ is an inner $k$-form of $\SL_n$; thus, $G$ is isomorphic to $\SL_{n'} \fD$ for some central division algebra $\fD$ over $k$ of index $d = n/n'$. Similarly, over $k_v$, $G$ is isomorphic to $\SL_{n_v} \fD_v$ for some central division algebra $\fD_v$ over $k_v$ of index $d_v = n / n_v$.
Recall that $T$ is the finite set of places $v \in V_f$ where $G$ does not split over $k_v$, and let $T'$ be the finite set of places $v \in V_f$ where $P_v$ is not a hyperspecial parahoric; of course, $T \subset T'$. The goal of this section is to show that $T'$ is empty. 

According to section \ref{indexinner}, we have
$$\# \rH^1(k,C)_\xi \leq \tilde{n} \cdot \prod_{v \in T} d_v,$$
with $d_v \geq 2$ if $v \in T$. 
Also, as we noted at the begining of section \ref{kisQ}, 
$$\# \Xi_{\Theta_v} \leq n_v \textrm{ if $v \in T$,} \quad \# \Xi_{\Theta_v} \leq r+1 = n \textrm{ if $v \in T'$,} \quad \# \Xi_{\Theta_v} = 1 \textrm{ otherwise.}$$
Combined with \eqref{volume} and \eqref{index}, we obtain
\begin{align}
\mu(G(k_{v_0})/ \Gamma) &\geq \tilde{n}^{-1} V_n \cdot \prod_{v \in T} d_v^{-1} \cdot \prod_{v \in T} n_v^{-1} \cdot \prod_{v \in T' - T} n^{-1} \cdot \prod_{v \in V_f} e(P_v) \notag \\
					&= \tilde{n}^{-1} V_n \cdot \prod_{v \in T'} n^{-1} \cdot \prod_{v \in V_f} e(P_v). \label{BQ}
\end{align}

Recall that for any $v \in V_f$, $e(P_v) > 1$. If $v \in T$, then according to \cite[remark 2.11]{Prasad89}, we have 
$$e(P_v) \geq (q_v-1)  q_v^{\frac{1}{2}(n^2 - n^2 d_v^{-1} -2)} \geq (q_v-1) q_v^{\frac{1}{4}n^2 - 1}.$$
Now if $T$ is not empty, then by looking at the Hasse invariant of $\fD$, it appears that $d_v \geq 2$ for at least two (finite) places. This means that $T$ has at least two elements, and using lemma \ref{ePv}, we see that if $n \geq 4$,
$$\prod_{v \in T} n^{-1} e(P_v) \geq (n^{-1} (2 -1) \cdot  2^{\frac{1}{4}n^2 - 1}) \cdot (n^{-1} (3-1) \cdot 3^{\frac{1}{4}n^2 - 1}) \geq 27.$$
If $n=3$, then actually $d_v = 3$ for at least two (finite) places, and 
$$\prod_{v \in T} n^{-1} e(P_v) \geq (n^{-1} (2 -1) \cdot  2^{\frac{1}{3}n^2 - 1}) \cdot (n^{-1} (3-1) \cdot 3^{\frac{1}{3}n^2 - 1}) = 8.$$
In particular, it is clear from $\eqref{BQ}$ that $\Gamma$ is not of minimal covolume. Hence it must be that $T$ is empty and $G$ splits everywhere. \bigbreak

On the other hand, if $v \in T' - T$, then $P_v$ is properly contained in a hyperspecial parahoric $H_v$. As discussed previously, there is a canonical surjection $H_v \to \SL_n(\ff_v)$, under which the image of $P_v$ is the proper parabolic subgroup $\overline{P}_v$ of $\SL_n(\ff_v)$ whose type consists of the vertices belonging to the type of $P_v$ in the Dynkin diagram obtained by removing the vertex corresponding to $H_v$ in the affine Dynkin diagram of $G(k_v)$. In particular, it follows that $[H_v: P_v] = [\SL_n(\ff_v) : \overline{P}_v]$ and thus using lemma \ref{indexPv},
$$e(P_v) = [H_v: P_v] \cdot e(H_v) \geq q_v^{n-1} \cdot e(H_v). $$
Of course, as $G$ splits everywhere, we have that $e(H_v)$ is equal to the corresponding factor $e(\SL_n(\Z_v)) = \prod_{i=2}^n \frac{1}{1 - q_v^{-i}}$ for $\SL_n(\Q_v)$. In consequence,
$$\frac{\mu(G(k_{v_0})/ \Gamma)}{\mu(\SL_n(\R)/ \SL_n(\Z))} \geq \frac{\tilde{n}^{-1} \prod_{v \in T'} n^{-1} \cdot \prod_{v \in V_f} e(P_v)}{\prod_{v \in V_f} e(\SL_n(\Z_v))} \geq \tilde{n}^{-1} \prod_{v \in T'} (n^{-1} q_v^{n-1}) \geq 1$$
with equality only if $n = 4$, $T' = \{2\}$ and $\# \Xi_{\Theta_2} = 4$. Notice however that this bound is rather rough; by examining the types of the parahorics carefully, one obtains much better bounds. For example, to achieve $\# \Xi_{\Theta_v} = n$, $P_v$ must be an Iwahori subgroup, in which case $[H_v: P_v] \geq q_v^{(n^2-n)/2}$ in lemma \ref{indexPv}. This rules out the equality case above and thus $T'$ must be empty as well.

%%%%%%%%%%

\section{Conclusion}

As we have shown in section \ref{hyperspecial}, $G$ splits over $k_v$ for all $v \in V_f$ and thus for all $v \in V$. As before, let $\fD$ be a central division algebra over $k$ $(= \Q)$ of degree $d$ such that $G \cong \SL_{n'}(\fD)$ over $k$. Now since $G$ splits at all places, we have for any $v \in V$ that $G(k_v) \cong \SL_n(k_v)$, or in other words, that the group of elements of reduced norm 1 in $\rM_{n'}(\fD) \ot_k k_v$ is isomorphic to $\SL_n(k_v)$. This implies that $\rM_{n'}(\fD) \ot_k k_v \cong \rM_n(k_v)$, i.e.~$\fD_v = \fD \ot_k k_v$ splits over $k_v$. It then follows from the Albert–Brauer–Hasse–Noether theorem that $\fD = k$ and in turn $G(k) \cong \SL_n(k)$ and $G$ is split over $k$. From hereon, we will thus identify $G$ with $\SL_n$ through this isomorphism, to be denoted $\eta$. 

Since each parahoric $P_v$ is hyperspecial, for each $v \in V_f$ there exists $g_v \in \GL_n(\Q_v)$ such that $g_v P_v g_v^{-1} = \SL_n(\Z_{v})$. As the family $\{P_v\}$ is coherent, we may assume that $g_v = 1$ except for finitely many $v \in V_f$. In this way, $g = (1, (g_v)_{v \in V_f})$ determines an element of the adele group $\GL_n(\A)$. The class group of $\GL_n$ over $\Q$ is trivial \cite[ch.~8]{PlatonovRapinchuk94}, therefore
$$\GL_n(\A) = (\GL_n(\R) \times \prod_{v \in V_f} \GL_n(\Z_v)) \cdot \GL_n(\Q),$$
and we can write $g = (1, (g'_v h)_{v \in V_f})$ for $g'_v \in \GL_n(\Z_v)$ and $h \in \GL_n(\Q)$. In consequence, $hP_vh^{-1} = g_v'^{-1}\SL_n(\Z_v) g'_v = \SL_n(\Z_v)$, and thus 
$$h \Lambda h^{-1} = h\SL_n(\Q)h^{-1} \cap \prod_{v \in V_f} h P_v h^{-1} = \SL_n(\Q) \cap \prod_{v \in V_f} \SL_n(\Z_v) = \SL_n(\Z).$$
In turn, $h \Gamma h^{-1} = \SL_n(\Z)$, as $\SL_n(\Z)$ (or equivalently $\Lambda$) is its own normalizer in $\SL_n(\R)$. One way to obtain this fact is using Rohlfs' exact sequence (see section \ref{indexbound}). Indeed, clearly $C(k_{v_0}) = C(k) \cap \Lambda$, and on the other hand, since $\Lambda$ is given by hyperspecial parahorics, we may identify
$$\rH^1(k,C)'_\Theta = \{ x \in \Q^\times / \Q^{\times n} \mid v(x) \in n\Z \textrm{ for $v \in V_f$, and $x \in \R^{\times n}$} \} = \{1\}.$$
Hence $\Gamma / \Lambda$ is trivial as claimed. 

Finally, retracing our identifications, we find that $\SL_n(\Z)$ is the image of $\Gamma$ under the automorphism $\sigma: \SL_n(\R) \xrightarrow{\iota} G(k_{v_0}) \xrightarrow{\eta} \SL_n(\R) \xrightarrow{c_h} \SL_n(\R)$ of $\SL_n(\R)$ (here $c_h$ denotes conjugation by $h$). This concludes the proof of the 
\begin{thm*}
Let $n \geq 3$ and let $\Gamma$ be a lattice of minimal covolume for some (any) Haar measure in $\SL_n(\R)$. Then $\sigma(\Gamma) = \SL_n(\Z)$ for some (algebraic) automorphism $\sigma$ of $\SL_n(\R)$. 
\end{thm*}

%%%%%%%%%%
\pagebreak 
\appendix

\section{Appendix: Bounds for sections \ref{kisQ} through \ref{hyperspecial}}

\begin{lem} \label{dirichletunits}
Let $k$ be a totally real number field of degree $m$ and let $l$ be a quadratic extension of $k$ of signature $(2m_1,m_2)$, so that $m = m_1 + m_2$. Let $n \in \N$ and set $l_0 = \{ x \in l^\times \mid N_{l/k}(x) \in k^{\times n}\}$ and $l_n = \{ x \in l^\times \mid w(x) \in n\Z \textrm{ for each normalized finite place $w$ of $l$}\}$. Then 
$$\# \left( (l_n \cap l_0) / l^{\times n} \right) \leq \# \left( \mu(l) / \mu(l)^n \right) \cdot \tilde{n}^{m-1} n^{m_1} \cdot \# \cC_n,$$
where $\mu(l)$ is the group of roots of unity of $l$, $\tilde{n} = 1$ or $2$ depending if $n$ is odd or even, and $\cC_n$ is the $n$-torsion subgroup of the class group $\cC$ of $l$. 

Moreover, if $N_{l/k}$ is surjective from $U_l$ onto $U_k / \{\pm 1\}$, then
$$\# \left( (l_n \cap l_0) / l^{\times n} \right) \leq \# \left( \mu(l) / \mu(l)^n \right) \cdot n^{m_1} \cdot \# \cC_n.$$
\end{lem}
\begin{proof}
According to \cite[prop.~0.12]{BorelPrasad89}, there is an exact sequence
$$1 \to U_l/U_l^n \to l_n / l^{\times n} \to \cC_n \to 1,$$
where $U_l$ denotes the group of units of the ring of integers of $l$, and $\cC_n$ is the $n$-torsion subgroup of the class group $\cC$ of $l$. Intersecting with $l_0 / l^{\times n}$ yields
$$\# \left( (l_n \cap l_0) / l^{\times n} \right) \leq \#\left( (U_l \cap l_0) / U_l^n \right) \cdot \# \cC_n.$$

Dirichlet's units theorem states that $U_l$ is the internal direct product $F_l \times \mu(l)$ of $F_l$, the free abelian subgroup of $U_l$ (of rank $2m_1+m_2 -1$) generated by some system of fundamental units, and $\mu(l)$, the subgroup of roots of unity in $l^\times$. Since $\mu(l) \subset l_0$, we also have that $U_l \cap l_0$ is the internal direct product of $F_l \cap l_0$ and $\mu(l)$. Additionally, it is clear that under this identification, $U_l^n$ corresponds to the subgroup $F_l^n \times \mu(l)^n$ of $(F_l \cap l_0) \times \mu(l)$. In consequence, 
$$\#\left( (U_l \cap l_0) / U_l^n \right) = \#\left( (F_l \cap l_0) / F_l^n \right) \cdot \# \left( \mu(l) / \mu(l)^n \right),$$
and it remains to study $(F_l \cap l_0)/ F_l^n$; to this end, we switch to additive notation. 

We write $L$ for the free abelian group $U_l / \mu(l)$ (canonically isomorphic to $F_l$) in additive notation, and $M$ for its free subgroup $U_k / \{ \pm 1\}$ (of rank $m-1$) consisting of units lying in $k$. The norm $N_{l/k}$ induces a map $N: L \to M$, and in turn a map $L/ nL \to M /nM$ also denoted by $N$, whose kernel $L_0 / nL$ corresponds precisely to $(F_l \cap l_0) / F_l^n$. In other words, the sequence
$$0 \to L_0 / nL \to L / nL \xrightarrow{N} M / nM$$
is exact. It is clear that $\# (L/nL) = n^{2m_1 + m_2 -1}$ and $\#(M/nM) = n^{m-1}$. If $N$ is surjective, then it follows that $\# (L_0 /nL) = n^{m_1}$. In any case, we have $2M \subset N(L)$ hence we may write
$$\# \left(\frac{N(L)+nM}{nM} \right) = \# \left( \frac{N(L) + nM}{2M + nM} \right) \cdot \# \left( \frac{2M +nM}{nM} \right).$$
As $2M + nM = \tilde{n}M$, we have $\# \left( \frac{2M +nM}{nM} \right) = \left( \frac{n}{\tilde{n}} \right)^{m-1}$ and the lemma follows. 
\end{proof}

\begin{lem} \label{ePv}
The function $\N \times \N \to \R$ defined by $E(n,q) = n^{-1} \cdot (q-1) q^{n^2 /4 -1}$ is increasing in both $n$ and $q$ provided $n, q \geq 2$. In consequence, $n^{-1} \cdot (q-1) q^{n^2 /4 -1} > 1$ provided $n \geq 4$. Similarly, $n^{-1} \cdot (q-1) q^{n^2 /3 -1} > 1$ provided $n \geq 3$.  
\end{lem}
\begin{proof}
We compute, for $n, q \geq 2$,
$$\frac{E(n,q+1)}{E(n,q)} = \frac{q (q+1)^{\frac{1}{4}n^2 - 1}}{(q-1) q^{\frac{1}{4}n^2 -1}} = \frac{q^2 (q+1)^{\frac{1}{4}n^2}}{(q^2 -1) q^{\frac{1}{4}n^2}} > 1.$$
and
$$\frac{E(n+1,q)}{E(n,q)} = \frac{n}{n+1} \cdot q^{\frac{1}{4}(2n+1)} \geq \frac{2}{3} \cdot 2^{\frac{5}{4}} > 1.$$
Thus $E$ is strictly increasing in $n$ and $q$ if $n, q \geq 2$, and $E(4,2) = 2$. The proof of the second inequality is analogous. 
\end{proof}

\begin{lem} \label{ePv2}
Let $n, q \in \N$ with $q \geq 2$. Then $\tilde{n}^{-1} \cdot (q+1)^{-1} q^{\ceil{(n+1)/2}} > 1$ provided $n \geq 3$. 
\end{lem}
\begin{proof}
Observe that $E(n,q) = \frac{q^{\ceil{(n+1)/2}}}{(q+1) \tilde{n}}$ is increasing in $n$ and strictly increasing in $q$, as
$$\frac{E(n+1,q)}{E(n,q)} = \frac{\tilde{n}}{\widetilde{n+1}} q^{2 - \tilde{n}} \geq 1$$
and
$$\frac{E(n,q+1)}{E(n,q)} = \frac{(q+1)(q+1)^{\ceil{(n+1)/2}}}{(q+2) q^{\ceil{(n+1)/2}}} = \frac{(q^2+2q +1)(q+1)^{{\ceil{(n+1)/2}}-1}}{(q^2 + 2q)q^{\ceil{(n+1)/2}-1}} > 1.$$
Finally $E(3,2) = \frac{4}{3}$. 
\end{proof}

\begin{lem} \label{Mout}
The function $M: \N \times \N \rightarrow \R$ defined by
$$M(m,n) = \frac{1}{100 \tilde{n}^{m}} \left( \frac{12}{\pi} \right)^{2m} \left( \frac{m^m}{m!} \right)^{n^2 -5} \left( \prod_{i=1}^{n-1} \frac{i!}{(2\pi)^{i+1}} \right)^{m-1} n^{-1}$$
(where $\tilde{n} = 1$ or $2$ if $n$ is odd or even) is strictly increasing in both $m$ and $n$, provided $m \geq 2$ and $n \geq 6$. 
\end{lem}
\begin{proof}
For $F$ a function of two integer variables $m$ and $n$, we denote $\partial_m F$ (resp.~$\partial_n F$) the function defined by $\partial_m F(m,n)= \frac{F(m+1,n)}{F(m,n)}$ (resp.~$\partial_n F(m,n)= \frac{F(m,n+1)}{F(m,n)}$). In order to show that $M$ increases in $m$ (resp.~in $n$), we intent to show that $\partial_m M > 1$ (resp.~$\partial_n M > 1$). 

We have
\begin{align*}
\partial_m M(m,n) &= \frac{144}{\pi^2 \tilde{n}} \left(\frac{(m+1)^m}{m^m}\right)^{n^2-5} \cdot \prod_{i=1}^{n-1} \frac{i!}{(2\pi)^{i+1}}\\
\partial_n M(m,n) &= \left(\frac{\tilde{n}}{\widetilde{n+1}}\right)^{m} \cdot \frac{n}{n+1} \cdot \left(\frac{m^m}{m!}\right)^{2n+1} \left(\frac{n!}{(2\pi)^{n+1}}\right)^{m-1}
\end{align*}
and thus
$$\partial_m (\partial_n M)(m,n) = \partial_n (\partial_m M)(m,n) = \frac{\tilde{n}}{\widetilde{n+1}} \cdot \left(\frac{(m+1)^m}{m^m}\right)^{2n+1} \cdot \frac{n!}{(2\pi)^{n+1}}.$$
Now if $m \geq 2$ and $n \geq 4$, then $\frac{(m+1)^m}{m^m} \geq \frac{9}{4}$ and we have
$$\partial_m (\partial_n M)(m,n) \geq \frac{1}{2} \left( \frac{9}{4} \right)^{2n+1} \frac{n!}{(2 \pi)^{n+1}} = \frac{9}{16 \pi}  \cdot \left( \frac{81}{16 \pi} \right)^{n} \cdot \frac{n!}{2^n} \geq \frac{9}{16 \pi} \cdot \left( \frac{81}{16 \pi} \right)^{4} > 1.$$
This means that provided $m \geq 2$ and $n \geq 4$, $\partial_m M$ increases in $n$ and $\partial_n M$ increases in $m$. 

Finally, assuming $m \geq 2$ and $n \geq 6$ respectively, we have
\begin{align*}
\partial_m M(m,6) &= \frac{144}{2 \pi^2} \left(\frac{(m+1)^m}{m^m}\right)^{31} \cdot \prod_{i=1}^{5} \frac{i!}{(2\pi)^{i+1}} \geq \frac{144}{2\pi^2} \left(\frac{9}{4}\right)^{31} \cdot \prod_{i=1}^{5} \frac{i!}{(2\pi)^{i+1}} > 1 \\
\partial_n M(2,n) &= \left(\frac{\tilde{n}}{\widetilde{n+1}} \right)^2 \frac{n}{n+1} \cdot 2^{2n+1} \cdot \frac{n!}{(2\pi)^{n+1}} \geq \frac{3}{14} \cdot 2^{n} \frac{n!}{\pi^{n+1}} \geq \frac{3}{14} \cdot 2^{6} \cdot \frac{6!}{\pi^{7}}  > 1
\end{align*}
hence $\partial_m M(m,n) > 1$ and $\partial_n M(m,n) > 1$ provided $m \geq 2$ and $n \geq 6$, completing the proof. 
\end{proof}

\begin{tabl} \label{tableMout}
The table below contains some values of the function $M$ from lemma \ref{Mout}. 
\center{ \scalebox{0.60}{$
\begin{array}{c|cccccccc}
(n, m) & 1 & 2 & 3 & 4 & 5 & 6 & 7 & 8 \\ \hline
2 &  0.0364756 & 0.00337012 & 0.000276781 & 0.0000215771 & 1.63315 \times 10^{-6} & 1.21281 \times 10^{-7} & 8.88761 \times 10^{-9} & 6.44933 \times 10^{-10} \\
3 & 0.0486342 & 0.00231876 & 0.000177084 & 0.0000166585 & 1.76356 \times 10^{-6} & 2.01469 \times 10^{-7} & 2.42731 \times 10^{-8} & 3.04153 \times 10^{-9} \\
4 & 0.0182378 & 0.000214239 & 9.19392 \times 10^{-6} & 6.99962 \times 10^{-7} & 7.37412 \times 10^{-8} & 9.57798 \times 10^{-9} & 1.43998 \times 10^{-9} & 2.41175 \times 10^{-10} \\
5 & 0.0291805 & 0.000860260 & 0.000267434 & 0.000235765 & 0.000375160 & 0.000873531 & 0.00265357 & 0.00980934 \\
6 & 0.0121585 & 0.000715847 & 0.00162363 & 0.0185268 & 0.528020 & 27.1489 & 2107.97 & 221884. \\
7 & 0.0208432 & 0.0374453 & 11.9823 & 37981.0 & 4.41409\times 10^8 & 1.18530\times 10^{13} & 5.71337\times 10^{17} & 4.24155\times 10^{22} \\
8 & 0.00911891 & 0.556912 & 35451.1 & 4.88495 \times 10^{10} & 3.84324 \times 10^{17} & 9.29477 \times 10^{24} & 4.92580 \times 10^{32} & 4.65827 \times 10^{40} \\
9 & 0.0162114 & 685.655 & 2.23863\times 10^{11} & 3.83726\times 10^{21} & 6.20398\times 10^{32} & 4.26138\times 10^{44} & 8.04066\times 10^{56} & 3.19899\times 10^{69} \\
10 & 0.00729513 & 306071. & 9.29184 \times 10^{17} & 3.98641 \times 10^{32} & 2.82701 \times 10^{48} & 1.22281 \times 10^{65} & 1.87055 \times 10^{82} & 7.27033 \times 10^{99} \\
11 & 0.0132639 & 1.40574\times 10^{10} & 1.27888\times 10^{28} & 4.91209\times 10^{48} & 5.79785\times 10^{70} & 6.22507\times 10^{93} & 3.12510\times 10^{117} & 4.89869\times 10^{141}
\end{array}$}}
\end{tabl}

\begin{lem} \label{Mout2}
The function $M': \N \times \N \rightarrow \R$ defined by
$$M'(m,n) = \frac{1}{100 \tilde{n}^{m}} \left( \frac{12}{\pi} \right)^{2m} \left( A^m E \right)^{n^2 -5} \left( \prod_{i=1}^{n-1} \frac{i!}{(2\pi)^{i+1}} \right)^{m-1} n^{-1}$$
(where $\tilde{n} = 1$ or $2$ if $n$ is odd or even, and $A = 29.534^{\frac{1}{2}}$, $E = e^{-4.13335}$) is strictly increasing in both $m$ and $n$, provided $m \geq 4$ and $n \geq 4$. Moreover, $M'(m,n)$ is strictly increasing in $m$ provided $n \geq 3$. 
\end{lem}
\begin{proof}
In order to show that $M'$ increases in $m$ (resp.~in $n$), we intend to show that $\partial_m M > 1$ (resp.~$\partial_n M > 1$); the notation is as in lemma \ref{Mout}. 

We have
\begin{align*}
\partial_m M'(m,n) &= \frac{144}{\pi^2 \tilde{n}} \cdot A^{n^2-5} \cdot \prod_{i=1}^{n-1} \frac{i!}{(2\pi)^{i+1}}\\
\partial_n M'(m,n) &= \left(\frac{\tilde{n}}{\widetilde{n+1}}\right)^{m} \left( A^m E \right)^{2n+1} \left(\frac{n!}{(2\pi)^{n+1}}\right)^{m-1} \left(\frac{n}{n+1}\right)
\end{align*}
and thus
$$\partial_m (\partial_n M')(m,n) = \partial_n (\partial_m M')(m,n) = \frac{\tilde{n}}{\widetilde{n+1}} \cdot A^{2n+1} \cdot \frac{n!}{(2\pi)^{n+1}}.$$
As clearly $A^2 > 2\pi$, we have (if $n \geq 3$)
$$\partial_m (\partial_n M')(m,n) > \frac{1}{2} \cdot A \cdot \frac{n!}{2\pi} > 1.$$
This means that $\partial_m M'$ increases in $n$ and $\partial_n M'$ increases in $m$. Assuming respectively $m \geq 1$ and $n \geq 4$, we have
\begin{align*}
\partial_m M'(m,3) 	&=  \frac{144}{\pi^2} \cdot A^{4} \cdot \frac{2}{(2\pi)^{5}} > 1 \\
\partial_n M'(4,n) 	&\geq \frac{1}{2^4} \cdot \left( A^4 E \right)^{2n+1} \cdot \frac{(n!)^3}{(2 \pi)^{3n+3}} \cdot \frac{4}{5} \\
				&\geq \frac{1}{2^4} \cdot \left( A^4 E \right)^{9} \cdot \frac{(6!)^3}{(2 \pi)^{21}} \cdot \frac{4}{5} > 1
\end{align*}
hence $\partial_m M'(m,n) > 1$ and $\partial_n M'(m,n) > 1$ provided $m \geq 4$ and $n \geq 4$. Moreover, $\partial_m M'(m,n) > 1$ if $n \geq 3$, completing the proof. 
\end{proof}

\begin{tabl} \label{tableMout2}
The table below contains some values of the function $M'$ from lemma \ref{Mout2}. 
\center{ \scalebox{0.57}{$
\begin{array}{c|cccccccc}
(n, m) & 1 & 2 & 3 & 4 & 5 & 6 & 7 & 8 \\ \hline
2 & 0.418729 & 0.0142379 & 0.000484124 & 0.0000164615 & 5.59732 \times 10^{-7} & 1.90323 \times 10^{-8} & 6.47149 \times 10^{-10} & 2.20047 \times 10^{-11} \\
3 & 2.80041 \times 10^{-6} & 7.27880 \times 10^{-6} & 0.0000189190 & 0.0000491740 & 0.000127813 & 0.000332209 & 0.000863474 & 0.00224433 \\
4 & 3.99708 \times 10^{-14} & 2.79970 \times 10^{-11} & 1.96100 \times 10^{-8} & 0.0000137356 & 0.00962086 & 6.73878 & 4720.08 & 3.30611 \times 10^{6} \\
5 & 1.84711 \times 10^{-23} & 2.62212 \times 10^{-16} & 3.72231 \times 10^{-9} & 0.0528412 & 750123. & 1.06486\times 10^{13} & 1.51165\times 10^{20} & 2.14591\times 10^{27} \\
6 & 1.68676 \times 10^{-35} & 2.85139 \times 10^{-23} & 4.82016 \times 10^{-11} & 81.4827 & 1.37743 \times 10^{14} & 2.32849 \times 10^{26} & 3.93621 \times 10^{38} & 6.65400 \times 10^{50} \\
7 & 4.80891 \times 10^{-49} & 1.09207 \times 10^{-29} & 2.48000 \times 10^{-10} & 5.63189\times 10^9 & 1.27896\times 10^{29} & 2.90442\times 10^{48} & 6.59571\times 10^{67} & 1.49783\times 10^{87} \\
8 & 2.65506 \times 10^{-65} & 6.66279 \times 10^{-38} & 1.67200 \times 10^{-10} & 4.19583 \times 10^{17} & 1.05293 \times 10^{45} & 2.64229 \times 10^{72} & 6.63074 \times 10^{99} & 1.66396 \times 10^{127} \\
9 & 4.52005 \times 10^{-83} & 1.88536 \times 10^{-45} & 7.86407 \times 10^{-8} & 3.28019\times 10^{30} & 1.36821\times 10^{68} & 5.70695\times 10^{105} & 2.38043\times 10^{143} & 9.92906\times 10^{180} \\
10 & 1.47804 \times 10^{-103} & 1.08376 \times 10^{-54} & 7.94662 \times 10^{-6} & 5.82681 \times 10^{43} & 4.27247 \times 10^{92} & 3.13276 \times 10^{141} & 2.29707 \times 10^{190} & 1.68431 \times 10^{239} \\
11 & 1.48182 \times 10^{-125} & 3.59121 \times 10^{-63} & 0.870337 & 2.10928\times 10^{62} & 5.11187\times 10^{124} & 1.23887\times 10^{187} & 3.00243\times 10^{249} & 7.27644\times 10^{311}
\end{array}$}}
\end{tabl}

\begin{tabl} \label{tableCout}
The table below contains some values of $C(m,n) = \left( 230 \tilde{n}^{m} \left( \frac{\pi}{12} \right)^{2m} V_n^{1-m} n \right)^{\frac{2}{n^2-5}}$. 
\center{ \scalebox{0.8}{$
\begin{array}{c|cccccccc}
(n, m) & 1 & 2 & 3 & 4 & 5 & 6 & 7 & 8 \\ \hline
3 & 6.87691 & 125.979 & 2307.81 & 42276.9 & 774473. & 1.41876\times 10^7 & 2.59904\times 10^8 & 4.76120\times 10^9 \\
4 & 2.40966 & 21.6241 & 194.053 & 1741.42 & 15627.4 & 140239. & 1.25850\times 10^6 & 1.12937\times 10^7 \\
5 & 1.54762 & 8.80582 & 50.1044 & 285.090 & 1622.14 & 9229.86 & 52517.2 & 298819. \\
6 & 1.40247 & 6.73460 & 32.3393 & 155.292 & 745.707 & 3580.86 & 17195.1 & 82570.5 \\
7 & 1.23838 & 4.82334 & 18.7864 & 73.1708 & 284.992 & 1110.01 & 4323.37 & 16839.0 \\
8 & 1.20619 & 4.19700 & 14.6037 & 50.8142 & 176.811 & 615.221 & 2140.69 & 7448.64 \\
9 & 1.13928 & 3.44306 & 10.4054 & 31.4468 & 95.0368 & 287.215 & 868.006 & 2623.24
\end{array}$}}
\end{tabl}

\begin{tabl} \label{tabledisck}
The table below contains the absolute value of the smallest discriminant $D_k$ of a totally real number field of degree $m$ (see for example \cite{Voight08} or \cite{JonesRoberts14}). 
\center{ \scalebox{1}{$
\begin{array}{c|ccccccccc}
m & 1 & 2 & 3 & 4 & 5 & 6 & 7 & 8 \\ \hline
\min D_k & 1 & 5 & 49 & 725 & 14641 & 300125 & 20134393 & 282300416
\end{array}$}}
\end{tabl}

\begin{tabl} \label{tablehl}
The tables below contains some values of $H(m) = \frac{ (A^2 B^{2m-2} E)^2 V_3^{m-1}}{3 \cdot \zeta(2) \cdot \zeta(3)}$ for $A =  25.465$, $B = 13.316$, $E= e^{-7.0667}$ if $m \geq 5$, and otherwise $H(m)$ is obtained from \eqref{Hout} using the smallest discriminant for the signature $(2,m-1)$ (see \cite{JonesRoberts14, Selmane99}). 
\center{ \scalebox{1}{$
\begin{array}{c|ccccccccc}
m & 2 & 3 & 4 & 5 & 6 & 7 & 8 & 9 \\ \hline
H(m) & 2.603 & 5.527 & 26.39 & 87.71 & 563.2 & 3616.4 & 23222.2 & 149118. 
\end{array}$}}
\center{ \scalebox{1}{$
\begin{array}{c|ccccccc}
m & 10 & 11 & 12 & 13 & 14 & 15 \\ \hline
H(m) & 9.58 \times 10^5 & 6.15\times 10^6 & 3.95\times 10^7 & 2.54\times 10^8 & 1.63\times 10^9 & 1.05\times 10^{10} 
\end{array}$}}
\end{tabl}

\begin{tabl} \label{tableDl}
The table below contains some values of $60.015^{2} \cdot 22.210^{2m-2} \cdot e^{\frac{-80.001}{H(m)}}$, where $H(m)$ is as in table \ref{tablehl}. 
\center{ \scalebox{0.9}{$
\begin{array}{c|ccccccccc}
m & 2 & 3 & 4 & 5 & 6 & 7 & 8 \\ \hline
D_l > & 8.05 \times 10^{-8} & 454.01 & 2.08\times 10^{10} & 8.57 \times 10^{13} & 9.13 \times 10^{16} & 5.08 \times 10^{19} & 2.55\times 10^{22}
\end{array}$}}
\center{ \scalebox{0.9}{$
\begin{array}{c|ccccccc}
m & 9 & 10 & 11 & 12 & 13 & 14 & 15 \\ \hline
D_l > &  1.26 \times 10^{25} & 6.23 \times 10^{27} & 3.07 \times 10^{30} & 1.52 \times 10^{33} & 7.48 \times 10^{35} & 3.69 \times 10^{38} & 1.82 \times 10^{41}
\end{array}$}}
\end{tabl}

\begin{lem} \label{Nn}
The function $N: \N \rightarrow \R$ defined by
$$N(n) = \frac{1}{100 \tilde{n}^2} \left( \frac{12}{\pi} \right)^{2} 40^{\frac{1}{4}(n^2-n-6)} n^{-1}.$$
(where $\tilde{n} = 1$ or $2$ if $n$ is odd or even) is strictly increasing provided $n \geq 2$. The same holds for
$$N'(n) = \tilde{n}^{-2} n^{-1} 5^{\frac{1}{4}(n^2 - n - 2)}.$$
\end{lem}
\begin{proof}
We compute
$$\frac{N(n+1)}{N(n)} = \frac{\tilde{n}^2}{\widetilde{n+1}^2} \cdot 40^{\frac{1}{2} n} \cdot \frac{n}{n+1} \geq \frac{1}{4} \cdot 40 \cdot \frac{2}{3} > 1.$$
The proof for $N'$ is analogous. 
\end{proof}

\begin{lem} \label{prodzeta}
$$\prod_{i=2}^\infty \zeta(i) < 2.3$$
\end{lem}
\begin{proof}
We have
\begin{align*}
\ln \prod_{i=9}^\infty \zeta(i) &= \sum_{i=9}^\infty \ln (1 + (\zeta(i) - 1)) \leq \sum_{i=9}^\infty (\zeta(i)-1) = \sum_{i=9}^\infty \sum_{j=2}^\infty \frac{1}{j^i} \\ 
					&= \sum_{j=2}^\infty \frac{1}{j^9} \sum_{i=0}^\infty \frac{1}{j^i} = \sum_{j=2}^\infty \frac{1}{j^9} \frac{j}{j-1} \leq 2 \sum_{j=2}^\infty \frac{1}{j^9} = 2(\zeta(9)-1);
\end{align*}
hence $\prod_{i=2}^\infty \zeta(i) \leq \exp(2\zeta(9)-2) \cdot \prod_{i=2}^8 \zeta(i) < 2.3$
\end{proof}

\begin{lem} \label{indexPv}
Let $P$ be a parabolic subgroup of $\SL_n(\F_q)$ and let $n_1, n_2, \dots, n_k$ be integers such that the complement of the type $\theta$ of $P$ in the Dynkin diagram of $\SL_n(\F_q)$ consists of $k- \# \theta$ connected components of respectively $n_1-1, n_2 -1, \dots, n_{k-\# \theta} -1$ vertices and $n_{k- \# \theta +1} = n_{k- \# \theta +2} = \dots = n_k = 1$. Then $[\SL_n(\F_q): {P}] \geq q^{\frac{1}{2}(n^2 - \sum_{i=1}^k n_i^2)}$. In particular, if $P$ is a proper parabolic subgroup, then $[\SL_n(\F_q): {P}] \geq q^{n-1}$. 
\end{lem}
\begin{proof}
Without loss of generality, we may assume that $P$ contains the subgroup $B$ of upper triangular matrices and that elements of $P$ are of the form
$$\newcommand*{\temp}{\multicolumn{1}{c|}{0}}
\left( \begin{array}{cccccccc}
\multicolumn{2}{c}{\multirow{2}{*}{$n_1$}} & \multicolumn{1}{|c}{*} & \cdots  & * & * & * & * \\
\multicolumn{2}{c}{} & \multicolumn{1}{|c}{*} & \cdots  & * & * & * & * \\ \cline{1-3}
0 & \multicolumn{1}{c|}{0} & n_2 & \multicolumn{1}{|c}{\cdots} & * & * & * & * \\ \cline{3-4}
\vdots & \vdots & \multicolumn{1}{c|}{\vdots} & \multicolumn{1}{c|}{\ddots} & \vdots & \vdots & \vdots & \vdots \\ \cline{4-7}
0 & 0 & 0 & \multicolumn{1}{c|}{\cdots} & \multicolumn{3}{c}{\multirow{3}{*}{$n_{k-1}$}} & \multicolumn{1}{|c}{*} \\
0 & 0 & 0 & \multicolumn{1}{c|}{\cdots} & \multicolumn{3}{c}{} & \multicolumn{1}{|c}{*} \\
0 & 0 & 0 & \multicolumn{1}{c|}{\cdots} & \multicolumn{3}{c}{} & \multicolumn{1}{|c}{*} \\ \cline{5-8}
0 & 0 & 0 & \cdots & 0 & 0 & \multicolumn{1}{c|}{0} & n_k
\end{array}\right)
$$
where $n_i$ indicates a block in $\GL_{n_i}(\F_q)$, $\ast$ indicates an arbitrary entry in $\F_q$, and the determinant of the whole matrix is 1. Hence
$$\#P = \frac{\prod_{j=1}^{n_1 - 1} (q^{n_1} - q^j) \dots \prod_{j=1}^{n_k - 1} (q^{n_k} - q^j) \cdot q^{\frac{1}{2}(n^2 - \sum_{i=1}^k n_i^2)}}{q-1}$$
and
\begin{align*}
\frac{\# \SL_n(\F_q)}{\#P} 	&= \frac{\prod_{j=0}^{n-1} (q^n - q^j)}{\prod_{j=0}^{n_1 - 1} (q^{n_1} - q^j) \dots \prod_{j=0}^{n_k - 1} (q^{n_k} - q^j) \cdot q^{\frac{1}{2}(n^2 - \sum_{i=1}^k n_i^2)}} \\
					&= \frac{q^{\frac{n(n-1)}{2}} \cdot \prod_{j=1}^n (q^j - 1)}{q^{\frac{1}{2}(n^2 - \sum_{i=1}^k n_i^2)} q^{\frac{1}{2}\sum_{i=1}^{k} n_i(n_i -1)} \cdot \prod_{j=1}^{n_1} (q^j - 1) \dots \prod_{j=1}^{n_k} (q^j - 1)} \\
					&= \frac{\prod_{j=1}^n (q^j - 1)}{\prod_{j=1}^{n_1} (q^j - 1) \dots \prod_{j=1}^{n_k} (q^j - 1)} \\
					&= q^{\frac{1}{2}(n(n-1) - \sum_{i=1}^k n_i(n_i -1))} \cdot \frac{\prod_{j=1}^{n_1} (q^j - 1) \cdot \prod_{j=1}^{n_2}(q^j - q^{-n_1}) \dots \prod_{j=1}^{n_k} (q^j - q^{- \sum_{i=1}^{k-1} n_i})} {\prod_{j=1}^{n_1} (q^j - 1) \dots \prod_{j=1}^{n_k} (q^j - 1)}.
\end{align*}
Of course, ${n(n-1) - \sum_{i=1}^k n_i(n_i -1)} = {(n^2 - \sum_{i=1}^k n_i^2)}$. Now the ratio in the right-hand side is clearly greater then 1, as, taken in order, each factor in the numerator is bigger than the corresponding one in the denominator. 

Finally, we observe that if $P$ is proper, $k \geq 2$ and $n^2 - \sum_{i=1}^k n_i^2 \geq 2n_1 n_2 \geq 2(n-1)$. 
\end{proof}

\pagebreak

%%%%%%%%%%

\bibliography{bibref.bib}

{{\bigskip \bigskip
  \footnotesize

\textsc{Department of Mathematics, University of California, San Diego \\
\indent 9500 Gilman Drive, La Jolla, CA 92093-0112}\par\nopagebreak
  \textit{E-mail address}:  \href{mailto:fthilman@ucsd.edu}{\texttt{fthilman@ucsd.edu}}
}}

\end{document}